\documentclass[twoside,a4paper,11pt]{amsart}

\usepackage{fullpage}
\usepackage{cmbright}
\usepackage{mathrsfs}
\usepackage[utf8]{inputenc}
\usepackage[T1]{fontenc}
\usepackage[english]{babel}
\usepackage{amsfonts}
\usepackage{amsthm}
\usepackage{amssymb}
\usepackage{mathcomp}
\usepackage{mathrsfs}
\usepackage[bb=boondox]{mathalfa}
\usepackage{graphicx}
\usepackage{epstopdf}
\usepackage[centertags]{amsmath}
\usepackage[pdftex]{hyperref}
\usepackage{vmargin}
\usepackage{tikz}
\usepackage{tikz-cd}
\usepackage{dsfont}
\usepackage{cases}
\usepackage{mathtools}
\usepackage{shuffle}
\usepackage{adjustbox}

\hypersetup{
    colorlinks,
    linkcolor=[rgb]{0.733, 0.149, 0.286},
    citecolor=[rgb]{0.733, 0.149, 0.286},
    urlcolor=[rgb]{0.733, 0.149, 0.286}
}

\title{A new approach to formal moduli problems}
\author{Brice Le Grignou \quad Victor Roca i Lucio}

\address{Brice Le Grignou,}
\email{\href{mailto:bricelegrignou@gmail.com}{bricelegrignou@gmail.com}}

\address{Victor Roca i Lucio, Ecole Polytechnique Fédérale de Lausanne, EPFL,
CH-1015 Lausanne, Switzerland}
\email{\href{mailto:victor.rocalucio@epfl.ch}{victor.rocalucio@epfl.ch}}

\date{\today}



\newtheorem{theoremintro}{Theorem}

\begin{document}

%new theorem
%------------------------------------------------------------------------------------------------------------------------
\theoremstyle{plain}
\newtheorem{theorem}{Theorem}%[section]
\newtheorem*{theorem*}{Theorem}
\newtheorem{lemma}{Lemma}%[section]
\newtheorem{proposition}{Proposition}%[section]
\newtheorem{assumption}{Assumption}%[section]
\newtheorem{corollary}{Corollary}%[section]

\theoremstyle{definition}
\newtheorem{definition}{Definition}%[section]

\theoremstyle{remark}
\newtheorem{remark}{\sc Remark}%[section]
\newtheorem{example}{\sc Example}%[section]
\newtheorem*{notation}{\sc Notation}
\newtheorem*{warning}{\sc Warning}

%new command
%------------------------------------------------------------------------------------------
\newcommand{\draftnote}[1]{\marginpar{\raggedright\textsf{\hspace{0pt} \tiny #1}}}
\newcommand{\ac}{{\scriptstyle \text{\rm !`}}}
\newcommand{\Ch}{\categ{Ch}}
\newcommand{\dgmod}{\Ch}

\newcommand{\qi}{\xrightarrow{ \,\smash{\raisebox{-0.65ex}{\ensuremath{\scriptstyle\sim}}}\,}}
\newcommand{\lqi}{\xleftarrow{ \,\smash{\raisebox{-0.65ex}{\ensuremath{\scriptstyle\sim}}}\,}}

\newcommand{\Cate}{\categ{Cat}_{\categ E}}
\newcommand{\Catesmall}{\categ{Cat}_{\categ E, \mathrm{small}}}
\newcommand{\Operade}{\Operad_{\categ E}}
\newcommand{\Operadecprime}{\Operad_{\categ E}}
\newcommand{\Operadesmall}{\Operad_{\categ E, \mathrm{small}}}
\newcommand{\eII}{\mathcal{I}}
\newcommand{\ecateg}[1]{\mathcal{#1}}
\newcommand{\cmonlax}{\categ{CMon}_\lax}
\newcommand{\cmonoplax}{\categ{CMon}_\oplax}
\newcommand{\cmonstrong}{\categ{CMon}_\strong}
\newcommand{\cmonstrict}{\categ{CMon}_\strict}
\newcommand{\cat}{\mathrm{cat}}
\newcommand{\lax}{\mathrm{lax}}
\newcommand{\oplax}{\mathrm{oplax}}
\newcommand{\strong}{\mathrm{strong}}
\newcommand{\strict}{\mathrm{strict}}
\newcommand{\pl}{\mathrm{pl}}
\newcommand{\Comonads}{\mathsf{Comonads}}
\newcommand{\Monads}{\mathsf{Monads}}
\newcommand{\Cats}{\mathsf{Cats}}
\newcommand{\Functors}{\mathsf{Functors}}
\newcommand{\Mor}{\mathsf{Mor}}
\newcommand{\Nat}{\mathsf{Nat}}
\newcommand{\sk}{\mathrm{sk}}
\newcommand{\TwoFun}{\text{2-Fun}}
\newcommand{\Ob}{\mathrm{Ob}}
\newcommand{\tr}{\mathrm{tr}}
\newcommand{\catch}{\mathsf{Ch}}
\newcommand{\categ}[1]{\mathsf{#1}}
\newcommand{\set}[1]{\mathrm{#1}}
\newcommand{\catoperad}[1]{\mathsf{#1}}
\newcommand{\operad}[1]{\mathcal{#1}}
\newcommand{\algebra}[1]{\mathrm{#1}}
\newcommand{\coalgebra}[1]{\mathrm{#1}}
\newcommand{\cooperad}[1]{\mathcal{#1}}
\newcommand{\ocooperad}[1]{\overline{\mathcal{#1}}}
\newcommand{\catofmod}[1]{{#1}\mathrm{-}\mathsf{mod}}
\newcommand{\catofcog}[1]{#1\mathrm{-}\mathsf{cog}}
\newcommand{\catcog}[1]{\categ{dg}\text{-}{#1}\text{-}\categ{cog}}
\newcommand{\catalg}[1]{\categ{dg}\text{-}{#1}\text{-}\categ{alg}}
\newcommand{\catofcolcomod}[1]{\mathsf{Col}\mathrm{-}{#1}\mathrm{-}\mathsf{comod}}
\newcommand{\catofcoalgebra}[1]{{#1}\mathrm{-}\mathsf{cog}}
\newcommand{\catofalg}[1]{\operad{#1}\mathrm{-}\mathsf{alg}}
\newcommand{\catofalgebra}[1]{{#1}\mathrm{-}\mathsf{alg}}
\newcommand{\mbs}{\mathsf{S}}
\newcommand{\catocol}[1]{\mathsf{O}_{\set{#1}}}
\newcommand{\catoftrees}{\mathsf{Trees}}
\newcommand{\cattcol}[1]{\catoftrees_{\set{#1}}}
\newcommand{\catcorcol}[1]{\mathsf{Corol}_{\set{#1}}}
\newcommand{\Einfty}{\mathcal{E}_{\infty}}
\newcommand{\nuEinfty}{\mathcal{nuE}_{\infty}}
\newcommand{\Fun}[3]{\mathrm{Fun}^{#1}\left(#2,#3\right)}
\newcommand{\III}{\operad{I}}
\newcommand{\treeoperad}{\mathbb{T}}
\newcommand{\treemodule}{\mathbb{T}}
\newcommand{\core}{\mathrm{Core}}
\newcommand{\forget}{\mathrm{U}}
\newcommand{\treemonad}{\mathbb{O}}
\newcommand{\cogcomonad}[1]{\mathbb{L}^{#1}}
\newcommand{\cofreecog}[1]{\mathrm{L}^{#1}}

\newcommand{\barfunctor}[1]{\mathrm{B}_{#1}}
\newcommand{\baradjoint}[1]{\mathrm{B}^\dag_{#1}}
\newcommand{\cobarfunctor}[1]{\mathrm{C}_{#1}}
\newcommand{\cobaradjoint}[1]{\mathrm{C}^\dag_{#1}}
\newcommand{\Operad}{\mathsf{Operad}}
\newcommand{\coOperad}{\mathsf{coOperad}}

\newcommand{\Aut}[1]{\mathrm{Aut}(#1)}

\newcommand{\verte}[1]{\mathrm{vert}(#1)}
\newcommand{\edge}[1]{\mathrm{edge}(#1)}
\newcommand{\leaves}[1]{\mathrm{leaves}(#1)}
\newcommand{\inner}[1]{\mathrm{inner}(#1)}
\newcommand{\inp}[1]{\mathrm{input}(#1)}

\newcommand{\field}{\mathbb{K}}
\newcommand{\mbk}{\mathbb{K}}
\newcommand{\mbn}{\mathbb{N}}

\newcommand{\id}{\mathrm{Id}}
\newcommand{\ii}{\mathrm{id}}
\newcommand{\unit}{\mathds{1}}

\newcommand{\Lin}{Lin}

\newcommand{\BijC}{\mathsf{Bij}_{C}}

\newcommand{\kk}{\Bbbk}
\newcommand{\PP}{\mathcal{P}}
\newcommand{\C}{\mathcal{C}}
\newcommand{\Sy}{\mathbb{S}}
\newcommand{\Tree}{\mathsf{Tree}}
\newcommand{\treemod}{\mathbb{T}}
\newcommand{\Dend}{\Omega}
\newcommand{\aDend}{\Omega^{\mathsf{act}}}
\newcommand{\cDend}{\Omega^{\mathsf{core}}}
\newcommand{\cDendpart}{\cDend_{\mathsf{part}}}

\newcommand{\build}{\mathrm{Build}}
\newcommand{\col}{\mathrm{col}}

\newcommand{\HOM}{\mathrm{HOM}}
\newcommand{\Hom}[3]{\mathrm{hom}_{#1}\left(#2 , #3 \right)}
\newcommand{\ov}{\overline}
\newcommand{\otimeshadamard}{\otimes_{\mathbb{H}}}

\newcommand{\Aa}{\mathcal{A}}
\newcommand{\BB}{\mathcal{B}}
\newcommand{\CC}{\mathcal{C}}
\newcommand{\DD}{\mathcal{D}}
\newcommand{\EE}{\mathcal{E}}
\newcommand{\FF}{\mathcal{F}}
\newcommand{\II}{\mathbb{1}}
\newcommand{\RR}{\mathcal{R}}
\newcommand{\UU}{\mathcal{U}}
\newcommand{\VV}{\mathcal{V}}
\newcommand{\WW}{\mathcal{W}}
\newcommand{\AAA}{\mathscr{A}}
\newcommand{\BBB}{\mathscr{B}}
\newcommand{\CCC}{\mathscr{C}}
\newcommand{\DDD}{\mathscr{D}}
\newcommand{\EEE}{\mathscr{E}}
\newcommand{\FFF}{\mathscr{F}}

\newcommand{\PPP}{\mathscr{P}}
\newcommand{\QQQ}{\mathscr{Q}}

\newcommand{\QQ}{\mathcal{Q}}

\newcommand{\KKK}{\mathscr{K}}
\newcommand{\KK}{\mathcal{K}}

\newcommand{\ra}{\rightarrow}

\newcommand{\Ai}{\mathcal{A}_{\infty}}
\newcommand{\uAi}{u\mathcal{A}_{\infty}}
\newcommand{\uEinfty}{u\mathcal{E}_{\infty}}
\newcommand{\uAW}{u\mathcal{AW}}

\newcommand{\uAlg}{\mathsf{Alg}}
\newcommand{\nuAlg}{\mathsf{nuAlg}}
\newcommand{\cAlg}{\mathsf{cAlg}}

\newcommand{\ucAlg}{\mathsf{ucAlg}}
\newcommand{\Cog}{\mathsf{Cog}}
\newcommand{\nuCog}{\mathsf{nuCog}}
\newcommand{\uAWcog}{u\mathcal{AW}-\mathsf{cog}}

\newcommand{\uCog}{\mathsf{uCog}}
\newcommand{\cCog}{\mathsf{cCog}}
\newcommand{\ucCog}{\mathsf{ucCog}}
\newcommand{\cNilCog}{\mathsf{cNilCog}}
\newcommand{\ucNilCog}{\mathsf{ucNilCog}}
\newcommand{\NilCog}{\mathsf{NilCog}}

\newcommand{\Cocom}{\mathsf{Cocom}}
\newcommand{\uCocom}{\mathsf{uCocom}}
\newcommand{\NilCocom}{\mathsf{NilCocom}}
\newcommand{\uNilCocom}{\mathsf{uNilCocom}}
\newcommand{\Liealg}{\mathsf{Lie}-\mathsf{alg}}
\newcommand{\cLiealg}{\mathsf{cLie}-\mathsf{alg}}
\newcommand{\Alg}{\mathsf{Alg}}
\newcommand{\Linfty}{\mathcal{L}_{\infty}}
\newcommand{\CMC}{\mathfrak{CMC}}
\newcommand{\Tfree}{\mathbb{T}}

\newcommand{\Hinich}{\mathsf{Hinich} -\mathsf{cog}}

\newcommand{\Ccomod}{\mathscr C -\mathsf{comod}}
\newcommand{\Pmod}{\mathscr P -\mathsf{mod}}

\newcommand{\cCoop}{\mathsf{cCoop}}

\newcommand{\Set}{\mathsf{Set}}
\newcommand{\sSet}{\mathsf{sSet}}
\newcommand{\dgMod}{\mathsf{dgMod}}
\newcommand{\gMod}{\mathsf{gMod}}
\newcommand{\catOrd}{\mathsf{Ord}}
\newcommand{\catBij}{\mathsf{Bij}}
\newcommand{\catSmod}{\mbs\mathsf{mod}}
\newcommand{\EEtw}{\mathcal{E}\text{-}\mathsf{Tw}}
\newcommand{\OpBim}{\mathsf{Op}\text{-}\mathsf{Bim}}

\newcommand{\Palg}{\mathcal{P}-\mathsf{alg}}
\newcommand{\Qalg}{\mathcal{Q}-\mathsf{alg}}
\newcommand{\Pcog}{\mathcal{P}-\mathsf{cog}}
\newcommand{\Qcog}{\mathcal{Q}-\mathsf{cog}}
\newcommand{\Ccog}{\mathcal{C}-\mathsf{cog}}
\newcommand{\Dcog}{\mathcal{D}-\mathsf{cog}}
\newcommand{\uCoCog}{\mathsf{uCoCog}}

\newcommand{\Artinalg}{\mathsf{Artin}-\mathsf{alg}}

\newcommand{\colim}{\mathrm{colim}}
\newcommand{\Def}{\mathrm{Def}}
\newcommand{\Bij}{\mathrm{Bij}}
\newcommand{\op}{\mathrm{op}}

\newcommand{\undern}{\underline{n}}
\newcommand{\dginterval}{{N{[1]}}}
\newcommand{\dgsimplex}[1]{{N{[#1]}}}

\newcommand{\cofree}{ T^c}
\newcommand{\Tw}{ Tw}
\newcommand{\End}{\mathcal{E}\mathrm{nd}}
\newcommand{\catEnd}{\mathsf{End}}
\newcommand{\coEnd}{\mathrm{co}\End}
\newcommand{\Mult}{\mathrm{Mult}}
\newcommand{\coMult}{\mathrm{coMult}}

\newcommand{\Lie}{\mathcal{L}\mathr{ie}}
\newcommand{\As}{\mathcal{A}\mathrm{s}}
\newcommand{\uAs}{\mathrm{u}\As}
\newcommand{\coAs}{\mathrm{co}\As}
\newcommand{\Com}{\mathcal{C}\mathrm{om}}
\newcommand{\uCom}{\mathrm{u}\Com}
\newcommand{\Perm}{\catoperad{Perm}}
\newcommand{\uBE}{\mathrm{u}\mathcal{BE}}
\newcommand{\uBEs}{{\uBE}^{\mathrm s}}

\newcommand{\PD}{\mathrm{PD}}
\newcommand{\abs}{\mathrm{abs}}

\newcommand{\itemt}{\item[$\triangleright$]}
\newcommand{\gr}{\mathrm{gr}}

\newcommand{\poubelle}[1]{}

\newcommand{\catdgalg}[1]{\mathsf{dg}~{#1}\text{-}\mathsf{alg}}
\newcommand{\catdgcog}[1]{\mathsf{dg}~{#1}\text{-}\mathsf{cog}}
\newcommand{\Map}[3]{\mathrm{Map}_{#1}\left(#2, #3\right)}
\newcommand{\Ab}{\mathrm{Ab}}
\newcommand{\Res}{\mathrm{Res}}
\newcommand{\FMP}{\mathsf{FMP}}
\newcommand{\comp}{\circ}
\newcommand{\complete}{\mathsf{comp}}
\newcommand{\Psh}{\mathsf{Psh}}
\newcommand{\Cell}{\mathrm{Cell}}
\newcommand{\Art}{\mathrm{Art}}
\newcommand{\CoArt}{\mathsf{CoArt}}
\newcommand{\invqis}{{\left[\mathrm{Q.iso}^{-1}\right]}}
\newcommand{\invw}{{\left[W^{-1}\right]}}
\newcommand{\catcellalg}[1]{\Cell~{#1}\text{-}\mathsf{alg}}
\newcommand{\catartalg}[1]{\Art~{#1}\text{-}\mathsf{alg}}
\newcommand{\catcoartcog}[1]{\CoArt~{#1}\text{-}\mathsf{cog}}

\newcommand{\Victor}[1]{\textcolor{blue}{#1}}
\newcommand{\Brice}[1]{\textcolor{red}{#1}}

\maketitle

\begin{abstract}
The main goal of this paper is to introduce a framework for infinitesimal deformation problems, using new methods coming from operadic calculus. We construct an adjunction between infinitesimal deformation problems over some type of algebras and their Koszul dual algebras, in any characteristic. This adjunction is an equivalence if and only if some algebras are equivalent to their completions. We give a concrete homological criterion for it. It gives us a new proof of the celebrated Lurie--Pridham theorem, as well as of many other generalizations of it. Our methods are effective, meaning they directly produce point-set models for the algebras that encode infinitesimal deformation problems. 
\end{abstract}

\setcounter{tocdepth}{1}
\tableofcontents

\section*{Introduction}
The idea that every infinitesimal deformation problem is encoded by a differential graded Lie algebra has a long history. It dates back to the works of K. Kodaira and D. Spencer on the deformation of complex structure on a manifold \cite{KodairaSpencer58}, and slowly became a guiding principle with the insights of P. Deligne, V. Drinfeld and many others. We refer to \cite{Toen,CalaqueGrivaux} for more details. Nevertheless, it was only recently that J. Pridham, and later J. Lurie were able to translate it into a precise mathematical statement.

\begin{theorem*}[\cite{Lurie11,Pridham}]\label{thm: intro}
Let $\kk$ be a field of characteristic zero. There is an equivalence of $\infty$-categories between the category of formal moduli problems and the category of dg Lie algebras over $\kk$.
\end{theorem*}

Here the notion of a formal moduli problem corresponds to a formalization of the idea of an "infinitesimal deformation problem". A formal moduli problem is a functor from $\kk$-augmented dg Artinian $\kk$-algebras with residue field $\kk$ to spaces, which preserves certain homotopy pullbacks. Informally speaking, these dg Artinian $\kk$-algebras correspond to infinitesimal neighbourhoods of a $\kk$-point of an algebro-geometric object, for instance a moduli space. A formal moduli problem therefore encodes the formal neighbourhood of this $\kk$-point via a "functor of points" perspective.

\medskip

The Lurie--Pridham theorem has been since then generalized in many directions. For example, B. Hennion generalized it to the case where the base field is a connective dg unital commutative $\kk$-algebra $A$ with finite dimensional homology, see \cite{Hennion}. Formal moduli problems then correspond to formal neighbourhoods of a $\kk$-point in a algebro-geometric object over $A$, and these are encoded by dg Lie algebras in dg $A$-modules. More generally, J. Nuiten showed in \cite{Liealgebroids} that formal neighbourhoods of more general geometric objects inside a moduli space are encoded by dg Lie \textit{algebroids}. 

\medskip

D. Calaque, R. Campos and J. Nuiten showed in \cite{CCN19} that the fact that dg Lie algebras are the ones appearing on the other side of the correspondence is a particular instance of Koszul duality for operads. This idea can be traced back to the approach of J. Lurie in \cite[Sections 3 and 4]{Lurie11}. One can define formal moduli problems in a more general setting than just dg Artinian commutative algebras, defining them over Artinian algebras over any augmented operad $\mathcal{P}$. They showed that for certain operads satisfying a technical condition, called \textit{splendid} operads, the $\infty$-category of formal moduli problems over $\mathcal{P}$-algebras is equivalent to the $\infty$-category of dg algebras over a Koszul dual operad.

\medskip

Another direction of generalizations is given when one considers positive characteristic base fields. In this case, the naive analogue of the Lurie--Pridham theorem is false, as dg Lie algebras are notoriously badly behaved in positive characteristic. Some results were still available for infinitesimal deformation problems defined over simplicial Artinian algebras, see \cite[Section 5]{Pridham}. However, L. Brantner and A. Mathew showed in \cite{brantnermathew} that the $\infty$-category of formal moduli problems is equivalent to the $\infty$-category of algebras over a monad (defined $\infty$-categorically). Algebras over this monad are called \textit{partition Lie algebras}. L. Brantner, R. Campos and J. Nuiten later showed in \cite{pdalgebras} that this $\infty$-category can be obtained as the localization of a semi-model category of divided power algebras over an operad. This in particular gives \textit{point-set} models for these partitions Lie algebras. There are two versions of these theorems, depending on whether one considers \textit{spectral} or \textit{derived} approaches; we will only focus on the spectral case from now on. Let us say a few words about the proofs of these theorems.

\medskip

\textbf{Pridham's method.} In Pridham's approach, infinitesimal deformation problems are defined at the point-set level, as left exact functors from dg strictly Artinian algebras to simplicial sets. He then endows them with a model structure generated by small extension. In order to relate them to dg Lie algebras, the idea is to use Koszul duality at the model categorical level. There is an adjunction between dg Lie algebras and conilpotent dg cocommutative coalgebras, given by the Chevalley-Eilenberg complex and its adjoint functor. This was already used by V. Hinich in \cite{Hinich} to transfer the classical model structure on dg Lie algebras, where weak-equivalences are given by quasi-isomorphisms and fibrations by degree-wise surjections, to conilpotent dg cocommutative coalgebras and obtain a Quillen equivalence, which was  interpreted in terms of derived deformation theory. Pridham uses the anti-equivalence, given by linear duality, between dg pro-Artinian algebras and conilpotent dg cocommutative coalgebras to relate dg Lie algebras with left exact functors and obtain a Quillen equivalence. About this topic, see also \cite{lazarev}.

\medskip

\textbf{Lurie's method.} The approach of Lurie relies on a general $\infty$-categorical machinery, which takes as input some version of Koszul duality and gives the equivalence; see \cite[Section 1]{Lurie11}. These arguments have been successively adapted in order to prove many generalizations. Let us illustrate this machinery in the original case.

\medskip

Instead of working with point-set models, Lurie defines the $\infty$-category of Artinian algebras as follows. Let $\kk$ denote the base field and let $S^n$ denote the chain complex with only $\kk$ in degree $n$. The $\infty$-category of Artinian algebras is the smallest sub-$\infty$-category of dg commutative algebras up to quasi-isomorphisms generated by $S^n$ for $n \geq 0$, and stable under pullbacks along maps $0 \longrightarrow S^n$ for $n > 0$. The $\infty$-category of \textit{formal moduli problems} is defined as the $\infty$-category of functors 

\[
\mathsf{FMP} \coloneqq \mathsf{Fun}^\ulcorner (\mathsf{Art}\text{-}\mathsf{alg}, \mathcal{S})~,
\]
\vspace{0.1pc}

which preserve pullbacks along maps $0 \longrightarrow S^n$ for $n > 0$, where $\mathcal{S}$ denotes the $\infty$-category of spaces. The key input now is the the Koszul duality that relates dg Lie algebras and dg commutative algebras, where one considers the linear dual of the Chevalley-Eilenberg complex.  However, here one views this object in the category of dg commutative algebras instead of in the category of dg pro-Artinian algebras. This construction induces a functor 

\[
\begin{tikzcd}
\widehat{\mathrm{CE}}: \mathsf{dg}~\mathsf{Lie}\text{-}\mathsf{alg}~[\mathrm{Q.iso}^{-1}] \arrow[r,"\mathrm{CE}^c"]
&\mathsf{dg}~\mathsf{Com}\text{-}\mathsf{coalg}~[\mathrm{Q.iso}^{-1}] \arrow[r,"(-)^*"]
&\mathsf{dg}~\mathsf{Com}\text{-}\mathsf{alg}^{\mathsf{op}}~[\mathrm{Q.iso}^{-1}]
\end{tikzcd}
\]
\vspace{0.1pc}

between the $\infty$-category of dg Lie algebras and the $\infty$-category of dg commutative algebras. Using that $\widehat{\mathrm{CE}}$ commutes with (homotopy) limits, one can deduce that it must have an $\infty$-categorical left adjoint, which is denoted by $\mathfrak{D}$. This induces a functor 

\[
\begin{tikzcd}[column sep=3.5pc,row sep=0.5pc]
\Psi: \mathsf{dg}~\mathsf{Lie}\text{-}\mathsf{alg}~[\mathrm{Q.iso}^{-1}] \arrow[r,]
&\mathsf{Fun}\left(\mathsf{Art}\text{-}\mathsf{alg},\mathcal{S}\right) \\
\mathfrak{g} \arrow[r,mapsto]
&\left[ A \mapsto \mathrm{Map}(\mathfrak{D}(A),\mathfrak{g}) \right]~.
\end{tikzcd}
\]
Here is where one needs to prove various key assertions:

\medskip

\begin{enumerate}
\item The functor $\mathfrak{D}(-)$ preserves (homotopy) pullback, and therefore the image of $\Psi$ lands in the $\infty$-category of formal moduli problems. 

\medskip

\item The functor $\Psi$ is conservative and preserves sifted (homotopy) colimits. Since it is accessible and preserves all limits, it admits a left adjoint $\Phi$.

\medskip

\item The functor $\Phi$ is fully faithful, that is, the unit $\mathrm{id} \qi \Psi \circ \Phi$ is a weak-equivalence.
\end{enumerate}

\medskip

If these assertions are true, then the adjunction $\Phi \dashv \Psi$ is an equivalence of $\infty$-categories. Indeed, using the fact that the unit of adjunction is a weak-equivalence and that $\Psi$ is conservative, we can deduce that the counit is also a weak-equivalence, and therefore that the adjunction is an equivalence. 

\medskip

The first point is in general non-trivial. The second point is somewhat easier, as it can be directly checked on $\mathrm{T} \circ \Psi$, where $\mathrm{T}$ is the tangent space of a formal moduli problem. The hardest assertion to show is the third one. It can be shown, using hypercoverings, that any formal moduli problem $F$ is weakly-equivalent to a the geometric realization of a simplicial object $X_\bullet$, where $X_n$ is pro-representable. Using the fact that $\Psi$ commutes with sifted (homotopy) colimits, one can then reduce to the pro-representable case. Using the fact that $\Psi$ commutes with filtered (homotopy) colimits, one can further reduce to the representable case. So the third point can be reduced to showing that for any Artinian algebra $A$,

\[
\eta_A: A \longrightarrow \widehat{\mathrm{CE}}~\mathfrak{D}(A)
\]
\vspace{0.1pc}

is a quasi-isomorphism. Even then, there are still quite non-trivial computations to be done.

\medskip

\textbf{Explaining the method.} The main goal of this paper is to introduce a new approach to formal moduli problems and generalize many of the known statements using it. Before stating the general results obtained in this paper, let us illustrate our method in the case of the Lurie--Pridham. There are two main ingredients. 

\medskip

\begin{enumerate}
    \item The first idea is to present the Koszul duality that relates Artinian algebras and cellular Lie algebras at the model categorical level, factoring it as a series of three Quillen adjunctions, where the first two always restrict to equivalences. 

\medskip

    \item The second idea is to leverage the presentation of the $\infty$-category of Lie algebras obtained in \cite{grignou2021algebraic} to reduce the proof to a simple, point-set comparison that involves the third Quillen adjunction above-mentioned.  
\end{enumerate}

\medskip

Recall that the $\infty$-category of formal moduli problems is defined as a sub-$\infty$-category of functors from the $\infty$-category of Artinian algebras to spaces. We present the $\infty$-category of Artinian algebras using $\mathcal{C}_\infty$-algebras, a point-set model for commutative algebras up to homotopy. Combining the operadic tools developped in \cite{linearcoalgebras} and in \cite{absolutealgebras}, we get the following sequence of Quillen adjunctions

\[
\begin{tikzcd}
    \mathcal{C}_\infty \text{-}\mathsf{alg}^{\mathsf{op}} \arrow[r, shift left=1.1ex, "(-)^\circ"{name=F}]           
    &\mathcal{C}_\infty \text{-}\mathsf{coalg} \arrow[l, shift left=.75ex, "(-)^*"{name=U}] \arrow[r, shift left=1.1ex, "\widehat{\Omega}"{name=A}]
    &\mathsf{dg}~\mathsf{abs}~\mathsf{Lie}\text{-}\mathsf{alg} \arrow[l, shift left=.75ex, "\widehat{\mathrm{B}}"{name=B}] \arrow[r, shift left=1.1ex, "\mathrm{Res}"{name=S}]
    &\mathsf{dg}~ \mathsf{Lie}\text{-}\mathsf{alg}~. \arrow[l, shift left=.75ex, "\mathrm{Abs}"{name=D}]  \arrow[phantom, from=F, to=U, , "\dashv" rotate=90]  \arrow[phantom, from=A, to=B, , "\dashv" rotate=-90]  \arrow[phantom, from=S, to=D, , "\dashv" rotate=90]
\end{tikzcd}
\]

Let us give more details.

\medskip

\begin{enumerate}
    \item The category of $\C_\infty$-algebras is the category of algebras over a cofibrant resolution of the commutative operad. It is endowed with a transferred model structure from chain complexes, where weak-equivalences are quasi-isomorphisms. Artinian algebras are the full sub-$\infty$-category of $\C_\infty$-algebras up to quasi-isomorphisms generated by $\{S^n\}_{n \geq 0}$, and stable under pullbacks along maps $0 \longrightarrow S^n$ for $n > 0$. 

    \medskip

    \item The category of $\C_\infty$-coalgebras is the category of coalgebras over a cofibrant resolution of the commutative operad. These are \textit{all, non-necessarily conilpotent} $\C_\infty$-coalgebras; we endow them with a transferred model structure from chain complexes, where weak-equivalences are \textit{quasi-isomorphisms}. These coalgebras and their model structure differ from those considered in \cite{Hinich, Pridham}. We define coArtinian coalgebras as the full sub-$\infty$-category of $\C_\infty$-coalgebras up to quasi-isomorphisms generated by $\{S^n\}_{n \leq 0}$, and stable under pushouts along maps $S^n \longrightarrow 0$ for $n < 0$. 

    \medskip
    
    \item The category of dg \textit{absolute} Lie algebras appears as the Koszul dual of these non-necessarily conilpotent $\mathcal{C}_\infty$-coalgebras. This is a new type of algebraic structure introduced in \cite{linearcoalgebras}, where infinite sums of operations have a well-defined image without assuming an underlying topology. These objects are encoded as algebras over a \textit{cooperad} instead of an operad. We endow them with a transferred model structure along the \textit{complete bar-cobar} adjunction $\widehat{\Omega} \dashv \widehat{\mathrm{B}}$ that relates them with $\C_\infty$-coalgebras, constructed in \cite{linearcoalgebras}. 

    \medskip
    
    \item Finally, consider dg Lie algebras with their transferred model structure from chain complexes, where weak-equivalences are quasi-isomorphisms. We define \textit{cellular Lie algebras} as the full sub-$\infty$-category of dg Lie algebras up to quasi-isomorphisms, generated by free Lie algebras $\{ \mathcal{L}\mathrm{ie} \circ S^n\}_{n \leq -1}$, and stable under pushouts along maps $\mathcal{L}\mathrm{ie} \circ S^n \longrightarrow 0$ for $n < -1$. 
    \medskip
\end{enumerate}

The first Quillen adjunction between $\C_\infty$-algebras and $\C_\infty$-coalgebras is given by the linear dual functor and its adjoint. By \cite{absolutealgebras}, the induced adjunction between $\infty$-categories restricts to an equivalence between the full sub-$\infty$-categories on objects with finite total dimensional homology ; this equivalence of $\infty$-categories restricts to an equivalence between Artinian algebras and coArtinian coalgebras. This allows us to define formal moduli problems purely in terms of $\C_\infty$-coalgebras, which can be viewed as formal derived affine schemes, similar to how non-necessirely conilpotent cocommutative coalgebras can be used to define formal affine schemes \cite{takeushi}. 

\medskip

The complete bar-cobar adjunction between $\C_\infty$-coalgebras and dg absolute Lie algebras is a Quillen equivalence by \cite{linearcoalgebras}. It sends coArtinian coalgebras to \textit{cellular absolute Lie algebras}: the full sub-$\infty$-category generated by free absolute Lie algebras $\{\widehat{\mathcal{L}}\mathrm{ie} \circ S^n\}_{n \leq -1}$, and stable under pushouts along maps $\widehat{\mathcal{L}}\mathrm{ie} \circ S^n \longrightarrow 0$ for $n < -1$. Thus, formal moduli problems are equivalent to the $\infty$-category of contravariant functors 

\[
\mathsf{FMP} \simeq \mathsf{Fun}^{\lrcorner}(\mathsf{Cell}~\mathsf{abs}~\mathsf{Lie}\text{-}\mathsf{alg}^\mathsf{op}, \mathcal{S})~,
\]
\vspace{0.1pc}

from cellular absolute Lie algebras to spaces which preserve pushouts along maps $\widehat{\mathcal{L}}\mathrm{ie} \circ S^n \longrightarrow 0$ for $n < -1$. By the results of \cite{grignou2021algebraic}, there is an equivalence of $\infty$-categories 

\[
\mathsf{dg}~\mathsf{Lie}\text{-}\mathsf{alg}~[\mathrm{Q.iso}^{-1}] 
\simeq \mathsf{Fun}^{\lrcorner}(\mathsf{Cell}~\mathsf{Lie}\text{-}\mathsf{alg}^\mathsf{op}, \mathcal{S})~,
\]
\vspace{0.1pc}

where on the right side we consider functors from cellular Lie algebras to spaces which preserve pushouts along maps $\mathcal{L}\mathrm{ie} \circ S^n \longrightarrow 0$ for $n < -1$. This presentation of the $\infty$-category of dg Lie algebras, called cellular presentation in  \cite{grignou2021algebraic} differs from the one that appears in \cite{Hennion} (called the cartesian presentation in \cite{grignou2021algebraic}).

\medskip

The Quillen adjunction $\mathrm{Ab} \dashv \mathrm{Res}$ induces an adjunction between $\infty$-categories whose left adjoint $\mathbb{L}\mathrm{Ab}$ sends cellular absolute Lie algebras to cellular dg Lie algebras. This yields another adjunction at the $\infty$-categorical level

\[
\begin{tikzcd}[column sep=5pc,row sep=5pc]
\mathsf{FMP} \arrow[r,"\Phi"{name=LDC},shift left=1.1ex ] 
&\mathsf{dg}~\mathsf{Lie}\text{-}\mathsf{alg}~[\mathrm{Q.iso}^{-1}] ~. \arrow[l,"\Psi"{name=TD},shift left=1.1ex] \arrow[phantom, from=TD, to=LDC, , "\dashv" rotate=90] 
\end{tikzcd}
\] 

 Notice that the adjunction goes, \textit{a priori}, in the other way as in Lurie's formalism. This adjunction is an equivalence of $\infty$-categories if and only if the $\infty$-categories of cellular Lie algebras and of cellular \textit{absolute} Lie algebras coincide. This can be checked at the point-set level, and by a simple \textit{dévissage} argument it amounts to showing that the inclusion 

\[
\eta: \bigoplus_{n \geq 0} \mathcal{L}\mathrm{ie}(n) \otimes_{\mathbb{S}_n} V^{\otimes n} \rightarrowtail \prod_{n \geq 0} \mathcal{L}\mathrm{ie}(n) \otimes_{\mathbb{S}_n} V^{\otimes n}~,
\]
\vspace{0.1pc}

of the free dg Lie algebra into the free completed dg Lie algebra is a quasi-isomorphism, where $V$ is any finite dimensional chain complex concentrated in degrees $\leq -1$. This follows from a straightforward observation: since there are only a finite number of non-zero terms in each degree, the direct sum coincides with the product and $\eta$ is an isomorphism. This is the only non-formal step in the proof.

\newpage

\textbf{General results.} Let $\kk$ be a field of any characteristic. On the one hand, we use the generalization of the homotopical operadic calculus to the positive characteristic case developped in \cite{premierpapier}. On the other hand, we use the presentations of $\infty$-category of algebras over a monad in terms of cellular objects obtained in \cite{presentationalgebras}, which are valid over any ring. 

\medskip

We consider formal moduli problems defined over Artinian $\mathcal{P}$-algebras, for $\mathcal{P}$ a dg operad. When working in positive characteristic, some extra hypothesis are required on $\mathcal{P}$ to have a meaningful homotopy theory on the category of $\mathcal{P}$-algebras. Now, if $\mathcal{P}$ is augmented, the augmentation allows us to endow $S^n$ with a trivial $\mathcal{P}$-algebra structure, and thus we can define the $\infty$-categories of Artinian $\mathcal{P}$-algebras and of formal moduli problems of $\mathcal{P}$-algebras. Without any loss of generality, we take a cofibrant replacement of $\mathcal{P}$ of the form $\Omega \C$, where $\C$ is a particularly well-behaved dg cooperad which we call \textit{quasi-planar}, see \cite{premierpapier}. This does not change the $\infty$-categories considered. Following the same steps as explained before, we get the following general result.

\begin{theoremintro}[Theorem \ref{Thm: the canonical adjunction}]
There is a canonical adjunction 
\[
\begin{tikzcd}[column sep=5pc,row sep=5pc]
\mathsf{FMP}_{\Omega\C} \arrow[r,"\Phi"{name=LDC},shift left=1.1ex ] 
&\mathsf{dg}~\C^*\text{-}\mathsf{alg}~[\mathrm{Q.iso}^{-1}]~,\arrow[l,"\Psi"{name=TD},shift left=1.1ex] \arrow[phantom, from=TD, to=LDC, , "\dashv" rotate=90] 
\end{tikzcd}
\] 
between the $\infty$-categories of formal moduli problems over $\Omega \C$-algebras and the $\infty$-category of dg $\C^*$-algebras. Furthermore, this adjunction is an equivalence if and only if the derived functor
\[
\begin{tikzcd}[column sep=4pc,row sep=5pc]
\mathsf{Cell}~\C^*\text{-}\mathsf{alg} \arrow[r,"\mathbb{L}\mathrm{Ab}"{name=TD}]
&\mathsf{Cell}~\C\text{-}\mathsf{alg}~.
\end{tikzcd}
\] 
is fully faithful. 
\end{theoremintro}

Again, notice that the adjunction here goes, in general, in the opposite way as the adjunctions previously constructed in the litterature (as it was only possible to construct them whenever they were an equivalence). Here the functor $\mathbb{L}\mathrm{Ab}$ compares cellular $\C^*$-algebras (divided power algebras over an operad) with cellular $\C$-algebras (over a cooperad). They are the non-absolute and the absolute versions of the same algebraic structure. We show that this functor is fully faithful if and only if the natural inclusion 

\[
\eta: \bigoplus_{n \geq 0} \left(\mathcal{C}(n)^* \otimes V^{\otimes n}\right)^{\mathbb{S}_n} \rightarrowtail \prod_{n \geq 0} \left(\mathcal{C}(n)^* \otimes V^{\otimes n}\right)^{\mathbb{S}_n}~.
\]

is a quasi-isomorphism, where $V$ is a finite dimensional dg module in degrees $\leq 0$. Let us point out that the map $\eta$ being a quasi-isomorphisms is a \textit{homotopy completeness} condition, analogous to the one introduced by J. Harper and K. Hess in \cite{HarperHess}. The condition in Theorem \ref{Thm: the canonical adjunction} can be reformulated as: cellular $\C^*$-algebras are homotopy complete. Over a characteristic zero field, it can also be reformulated as: the Goodwillie tower of the identity functor of dg $\C^*$-algebras converges on cellular objects. 

\medskip

For $\C$ $0$-reduced, meaning $\C(0) = 0$, we give a concrete homological criterion on $\C$ called \textit{temperedness} that implies that $\eta$ is indeed a quasi-isomorphism.

\begin{theoremintro}[Theorem \ref{thm: main theorem}]
Let $\C$ be a \textit{tempered} quasi-planar conilpotent dg cooperad. The canonical adjunction 
\[
\begin{tikzcd}[column sep=5pc,row sep=5pc]
\mathsf{FMP}_{\Omega\C} \arrow[r,"\Phi"{name=LDC},shift left=1.1ex ] 
&\mathsf{dg}~\C^*\text{-}\mathsf{alg}~[\mathrm{Q.iso}^{-1}]~,\arrow[l,"\Psi"{name=TD},shift left=1.1ex] \arrow[phantom, from=TD, to=LDC, , "\dashv" rotate=90] 
\end{tikzcd}
\] 

is an equivalence of $\infty$-categories.
\end{theoremintro}

Let us point out that the condition of $\C$ being tempered is equivalent over a zero characteristic field to that of $\Omega\C$ being \textit{splendid} in the sense of \cite{CCN19}. This gives a new interpretation of the splendid condition, which is not just a technical condition, but in fact comes from the comparison between non-absolute and absolute cellular algebras. 

\medskip

We also explain how to recover the formal moduli problem associated to a \textit{homotopy complete} dg $\C^*$-algebras via an explicit functor with values in simplicial sets, using the theory of mapping of coalgebras of \cite{grignou2022mapping} and the  integration functor for absolute partition $\mathcal{L}_\infty$-algebras constructed in \cite{lietheoryp}. 

\medskip

Finally, we give examples in Section \ref{Section: Examples}. In characteristic zero, our results recover those of Calaque--Campos--Nuiten in \cite{CCN19}, and in particular the Lurie--Pridham theorem, as explained before. Over a positive characteristic field, we obtain \textit{directly}, via our methods, point-set models that encode formal moduli problems of $\mathbb{E}_\infty$-algebras, which recovers the results of \cite{brantnermathew} and of \cite{pdalgebras}. We also recover the fact that formal moduli problems over $\mathbb{E}_k$-algebras are encoded by $\mathbb{E}_k$-algebras, where $\mathbb{E}_k$ is a model for the chains of the little $k$-disks operad. This was shown by Lurie in \cite[Section 3 and 4]{Lurie11}. Another example is given by permutative deformation problems in positive characteristic: we show that these are encoded by partition pre-Lie algebras. These are divided power algebras over the pre-Lie operad tensored with the linear dual of the surjections cooperad of \cite{pdalgebras}. This indicates that operadic deformation complexes possess a richer structure than just the pre-Lie product when working over a positive characteristic field. 

\medskip

Finally, let us point out that all the results of this paper extend \textit{mutatis mutandis} to coloured operads, but for simplicity's sake, we write them in the non-colored case. The presentation results of \cite{grignou2021algebraic} are valid over any ring, therefore the only obstruction to further generalizing these methods in mixed characteristic, or even over the integers is given by the need of a more general homotopical operadic calculus as the one developped in \cite{premierpapier}. 

\subsection*{Acknowledgements}
We would like to thank Lukas Brantner, Damien Calaque, Ricardo Campos, Geoffroy Horel, Joost Nuiten, Jon Pridham and Jérôme Scherer for stimulating discussions and helpful comments. We also thank Damien Calaque and Christina Kapatsori for helpful suggestions on the previous versions.

\subsection*{Conventions}
The base category we consider is the category of differential graded modules over a field $\kk$ of any characteristic, choosing the \textit{homological convention}. It forms a symmetric monoidal category, where one can consider operads and cooperads. We refer to \cite[Section 1]{premierpapier} for more precise recollections on this setup.

\medskip

Let $\mathcal{C}$ be a category and let $\mathrm{W}$ be a class of arrows in $\mathcal{C}$. We will denote $\mathcal{C}[\mathrm{W}^{-1}]$ the $\infty$-category obtained by localizing $\mathcal{C}$ at $\mathrm{W}$. When working at the $\infty$-categorical level, limits and colimits should be understood as meaning homotopy limits and colimits. We will avoid the notation "$[\mathrm{W}^{-1}]$" when considering $\infty$-categories defined by universal $\infty$-categorical properties; e.g.: Artinian algebras, coArtinian algebras, cellular algebras, etc.

\medskip

\section{Equivalent definitions of formal moduli problems of algebras over an operad}

\medskip

\subsection{First definitions}
For the rest of this section, let $\operad P$ be an augmented $\mathbb S$-projective dg operad $\PP$. An augmentation is the data of a morphism of dg operads $\eta: \PP \longrightarrow \III$. And by $\mathbb S$-projective we mean that its underlying dg $\mathbb S$-module is cofibrant for the projective model structure. This last hypothesis is necessary in order to make sense of the homotopy theory of dg $\mathcal{P}$-algebras since we are working over a field $\kk$ of any characteristic. 

\begin{notation}
Let $n$ be in $\mathbb{N}$. We denote by $S^n$ the dg module given by 
\[
(S^n)_m = 
\begin{cases}
\kk \text{ if } \text{ and } n=m;
\\
0 \text{ otherwise.}
\end{cases}
\]

Likewise, we will denote $D^n$ the dg module which is given by $\kk$ is degrees $n$ and $n-1$, where the differential is given by the identity map. 
\end{notation}

The augmentation morphism $\eta: \PP \longrightarrow \III$ induces a Quillen adjunction
\[
\begin{tikzcd}[column sep=7pc,row sep=3pc]
\mathsf{dg}~\PP\text{-}\mathsf{alg} \arrow[r, shift left=1.1ex, "\mathrm{indec}"{name=F}] 
&\mathsf{dg}~\mathsf{mod} ~, \arrow[l, shift left=.75ex, "\mathrm{triv}"{name=U}] \arrow[phantom, from=F, to=U, , "\dashv" rotate=-90]
\end{tikzcd}
\]

where the functor $\mathrm{indec}(-)$ is given by indecomposable elements of a dg $\PP$-algebra and the functor $\mathrm{triv}(-)$ endows any dg module with the structure of a trivial $\PP$-algebra. In particular, the dg modules $S^n, D^n$ can be endowed with a trivial dg $\PP$-algebra structure for all $n$ in $\mathbb{Z}$.

\medskip

Since $\mathcal{P}$ is a $\mathbb{S}$-projective dg operad, the category of dg $\mathcal{P}$-algebras can be endowed with a \textit{semi-model structure}, where weak-equivalences are given by quasi-isomorphisms and where fibrations are given by degree-wise epimorphisms. See \cite{Spitzweck} for more details.

\medskip

This semi-model structure presents the $\infty$-category of dg $\mathcal{P}$-algebras localized at the class of quasi-isomorphisms.

\begin{definition}[Artinian $\PP$-algebras]
The $\infty$-category of \textit{Artinian} $\PP$\textit{-algebras} is the smallest full subcategory of the $\infty$-category of dg $\PP$-algebras such that  

\medskip

\begin{enumerate}
\item it contains $S^n$ for any $n \geq 0$,

\medskip

\item for any Artinian $\PP$-algebra $A$ and any morphism $A \longrightarrow S^n$, where $n \geq 1$, the pullback $A \times_{S^n} 0$ is also Artinian.
\end{enumerate}
\end{definition}

\begin{remark}
Notice that the zero algebra is the terminal object in the category of dg $\mathcal{P}$-algebras, while the initial object is $\mathcal P(0)$.
\end{remark}

\begin{definition}[Formal moduli problems over $\PP$-algebras]
The $\infty$-category of \textit{formal moduli problem over} $\PP$-algebras is defined as the $\infty$-category of functors 
\[
F: \mathsf{Art}~\PP\text{-}\mathsf{alg}\longrightarrow \mathcal{S}
\]

from the $\infty$-category of Artinian $\PP$-algebras to the $\infty$-category of spaces satisfying the following conditions:

\begin{enumerate}
\item $F(0) \simeq \{*\}$, where $0$ is the zero $\PP$-algebra.

\medskip

\item $F$ preserves pullback diagrams of the form
\[
\begin{tikzcd}
A' \arrow[r] \arrow[d] \arrow[dr, phantom, "\lrcorner", very near start]
&0 \arrow[d] \\
A \arrow[r]
&S^n~,
\end{tikzcd}
\]

where $n \geq 1$.
\end{enumerate}

We denote this $\infty$-category by $\mathsf{FMP}_{\mathcal{P}}$.
\end{definition}

\subsection{Cofibrant replacements using quasi-planar cooperads} 
The definition of the $\infty$-category of formal moduli problems does not depend on a specific choice of $\mathbb{S}$-projective dg operad.

\begin{lemma}
Let $f: \mathcal{Q} \qi \mathcal{P}$ be a quasi-isomorphism between two $\mathbb{S}$-projective dg operads. It induces an equivalence of $\infty$-categories 
\[
\mathsf{FMP}_{\mathcal{Q}} \simeq \mathsf{FMP}_{\mathcal{P}}~.
\]

\end{lemma}

\begin{proof}
The quasi-isomorphism $\mathcal{Q} \qi \mathcal{P}$ induces a Quillen equivalence of semi-model structures 
\[
\begin{tikzcd}[column sep=7pc,row sep=3pc]
\mathsf{dg}~\PP\text{-}\mathsf{alg}  \arrow[r, shift left=1.1ex, "f^*"{name=F}] 
&\mathsf{dg}~\mathcal{Q}\text{-}\mathsf{alg}   ~, \arrow[l, shift left=.75ex, "f_!"{name=U}] \arrow[phantom, from=F, to=U, , "\dashv" rotate=90]
\end{tikzcd}
\]

since both dg operads are $\mathbb{S}$-projective (see \cite[Theorem 16.A]{Fresse}). Moreover, the $\infty$-functor $\mathbb R f^*$ derived from $f^*$ clearly preserves the objects $S^n$ for any $n \geq 0$; thus it induces an equivalence between the $\infty$-categories of Artinian $\mathcal{P}$-algebras and Artinian $\mathcal{Q}$-algebras.
\end{proof}

Therefore, up to quasi-isomorphisms of dg operads, one can choose the operad $\mathcal{P}$ to be cofibrant in the category of dg operads endowed with the semi-model structure of \cite[Chapter 12]{Fresse}. Furthermore, in order to apply the homotopical operadic calculus developed in \cite{premierpapier}, we can choose $\mathcal{P}$ to be of the form $\Omega \C$, where $\C$ is a \textit{quasi-planar} conilpotent dg cooperad. Let us recall the definition of a quasi-planar dg cooperads.

\begin{definition}[Quasi-planar dg cooperad, \cite{premierpapier}]
A \textit{pseudo-planar} conilpotent dg cooperad $\C$ is the data $(\C,\C_\pl,\varphi_\C)$ of a conilpotent dg cooperad $\C$ and graded planar conilpotent cooperad $\C_\pl$, together with a isomorphism $\varphi_\C$ of graded conilpotent cooperads
    \[
    \varphi_\C: \C_\pl \otimes \mathbb S \cong \C.
    \]
    It is called \textit{quasi-planar} if it is the colimit of a cocontinuous directed diagram of pseudo-planar dg conilpotent cooperads
    \[
    \operad I = \C^{(0)} \longrightarrow \C^{(1)}
    \longrightarrow \cdots \C^{(i)} \longrightarrow \cdots
    \]
    indexed by an ordinal $\alpha$, where for every $i, i+1 \in \alpha$, the map $\operad C^{(i)} \longrightarrow \operad C^{(i+1)}$ is an arity-wise degree-wise injection and such that
\begin{enumerate}
\item the decomposition map
\[
\begin{tikzcd}
\operad C^{(i+1)} \arrow[r]
&\treemod \operad C^{(i+1)} \arrow[r]
&\overline \treemod \operad C^{(i+1)}
\end{tikzcd}
\]
factors through $\overline \treemod \operad C^{(i)}$;
\item the restriction of the coderivation of $\operad C^{(i+1)}$ to $\operad C_\pl^{(i+1)} \otimes 1$ factors through
        
        \[
        \operad C_\pl^{(i+1)} \otimes 1 \longrightarrow \operad C_\pl^{(i+1)} \otimes 1 + \operad C^{(i)} \hookrightarrow \operad C^{(i+1)}.
        \]
        \vspace{0.1pc} 
In other words, the differential of $\gr_{i+1} ~\operad C = \left(\operad C_\pl^{(i+1)}/ \operad C_\pl^{(i)}\right) \otimes \mathbb S$ has the form $d_\pl \otimes \id_{\mathbb S}$.
\end{enumerate}
\end{definition}

In full generality, the operadic cobar construction $\Omega$ sends a conilpotent \textit{curved} cooperad $\C$ to a dg operad $\Omega \C$ whose underlying graded operad is the free operad $\treemod (s^{-1} \overline{\C})$ generated by the desuspension of coaugmentation ideal of $\C$. The differential is constructed using the coderivation, the decomposition maps and the curvature of $\C$. It admits a right adjoint, the curved bar functor $\mathrm{B}_{\mathrm{crv}}$, as shown in \cite{unitalalgebras}. See also \cite{premierpapier}. 

\medskip

If we restrict to conilpotent \textit{dg} cooperads $\C$, then $\Omega \C$ becomes an augmented dg operad. Furthermore, the restricted cobar $\Omega$ has a right adjoint, given by the (dg) bar construction $\mathrm{B}$, which is the one we will use in this paper. For any augmented dg operad $\operad P$, its bar construction is a conilpotent dg cooperad whose underlying graded cooperad is given by $\treemod (s\overline{\operad P})$, the cofree conilpotent cooperad generated by the suspension of the augmentation ideal of $\operad P$. Its differential is constructed using the derivation and the composition maps of $\operad P$. See for instance \cite[Chapter 6]{LodayVallette} for more details.

\begin{notation}
We will denote by $\operad E$ the Barratt-Eccles dg operad introduced in \cite{BergerFresse}. Recall that it is a particular choice of a $\mathbb{S}$-projective resolution of the commutative operad $\mathcal{C}\mathrm{om}$. For any dg operad $\operad P$, there is a quasi-isomorphism of dg operads $\operad E \otimes \operad P \qi \operad P$.
\end{notation}

\begin{proposition}[{\cite[Proposition 11, Proposition 12, Proposition 17]{premierpapier}}]\label{prop: cofibrant replacement for operads}
Let $\operad P$ be an augmented dg operad. There is a canonical structure of quasi-planar cooperad on the conilpotent dg cooperad $\mathrm{B}(\operad E \otimes \operad P)$. Its cobar construction $\Omega \mathrm{B}(\operad E \otimes \operad P)$ is cofibrant in the semi-model structure of dg operads. Finally, the canonical map
    \[
    \epsilon_{\operad P}: \Omega \mathrm{B}(\operad E \otimes \operad P) \qi \operad P
    \]
is a quasi-isomorphism of dg operads, and provides us with a cofibrant replacement for $\operad P$.
\end{proposition}

\begin{remark}
In \cite{premierpapier}, we considered $\mathrm{B}_{\mathrm{crv}}(\operad E \otimes \operad P)$, where $\mathrm{B}_{\mathrm{crv}}$ is the curved bar construction, instead of $\mathrm{B}(\operad E \otimes \operad P)$, in order to work in full generality. Nevertheless, the same arguments, \textit{mutatis mutandis}, also work for $\mathrm{B}(\operad E \otimes \operad P)$.
\end{remark}

Hence, by Proposition \ref{prop: cofibrant replacement for operads}, we can restrict to dg operads of the form $\Omega\C$ for any quasi-planar conilpotent dg cooperad. Notice that we can also assume $\C$ to be $0$-reduced since we assumed $\PP$ to be $0$-reduced. 

\begin{remark}
Over a field of characteristic zero, $\C$ can be assumed to be any conilpotent dg cooperad, as any dg operad admits a cofibrant resolution of the form $\Omega\C$, where $\C$ is a conilpotent dg cooperad.
\end{remark}

%-------------------------------------

\subsection{Strictly Artinian algebras}
From now on, let us fix a quasi-planar conilpotent dg cooperad $\C$. Our goal will be to describe the $\infty$-category of formal moduli problems over $\Omega \C$. Any dg operad of the form $\Omega \C$ is admissible, thus the semi-model category structure on dg $\Omega \C$-algebras actually forms a model category structure. Moreover, there is a bar-cobar adjunction relative to the universal twisting morphism $\iota: \C \longrightarrow \Omega \C$ 
\[
\begin{tikzcd}[column sep=5pc,row sep=3pc]
          \mathsf{dg}~\C\text{-}\mathsf{coalg} \arrow[r, shift left=1.1ex, "\Omega_{\iota}"{name=F}] & \mathsf{dg}~\Omega \C\text{-}\mathsf{alg}, \arrow[l, shift left=.75ex, "\mathrm{B}_{\iota}"{name=U}]
            \arrow[phantom, from=F, to=U, , "\dashv" rotate=-90]
\end{tikzcd}
\]

which becomes a Quillen equivalence when one considers the transferred model category structure along this adjunction on the category of dg $\C$-coalgebras. See \cite[Section 5]{premierpapier}.

\begin{definition}[Square-zero extension]
Let $f: B \longrightarrow A$ be a morphism of dg $\Omega \C$-algebras. It is a \textit{square-zero extension} if $f$ fits in a homotopy pullback 
\[
\begin{tikzcd}
S^n \arrow[r] \arrow[d]  \arrow[dr, phantom, "\lrcorner", very near start]
&B \arrow[d,"f"] \\
0\arrow[r]
&A~,
\end{tikzcd}
\]

for some $n \geq 0$. 
\end{definition}

\begin{definition}[Strictly Artinian algebra]
Let $A$ be a dg $\Omega \C$-algebra. It is \textit{strictly Artinian} if the terminal morphism $A \longrightarrow 0$ admits a factorization 
\[
A = A^{(n)} \longrightarrow \cdots \longrightarrow A^{(i)} \longrightarrow A^{(i-1)} \longrightarrow \cdots \longrightarrow A^{(0)} = 0~,
\]

where each morphism $A^{(i)} \longrightarrow A^{(i-1)}$ is a \textit{square-zero extension}.
\end{definition}

\begin{proposition}\label{Prop: Artinian = Strictly Artinian}
There is an equivalence of $\infty$-categories between the $\infty$-category of Artinian $\Omega \C$-algebras and the $\infty$-category of strictly Artinian $\Omega \C$-algebras. 
\end{proposition}

\begin{proof}
It can easily shown that any strictly Artinian $\Omega \C$-algebra is Artinian, see \cite[Example 2.3]{CCN19}. In order to show that any Artinian $\Omega \C$-algebra is quasi-isomorphic to a strictly Artinian $\Omega \C$-algebra, we follow the same steps as in \cite[Lemma 5.12]{CCN19}. It suffices to show that strictly Artinian $\Omega \C$-algebras are also closed under homotopy pullbacks of the form $0 \longrightarrow S^n$ for $n \geq 1$. Let $A$ be a strictly Artinian $\Omega \C$-algebra, we consider the following (homotopy) pullback 
\[
\begin{tikzcd}
Y \arrow[r] \arrow[d]  \arrow[dr, phantom, "\lrcorner", very near start]
&D^n \arrow[d] \\
\Omega_\iota \mathrm{B}_{\iota} A\arrow[r]
&S^n~.
\end{tikzcd}
\]

We want to show that the map $Y \longrightarrow \Omega_\iota \mathrm{B}_{\iota} A$ is naturally quasi-isomorphic to a square-zero extension $B \longrightarrow A$ with kernel $S^n$. Using the bar-cobar adjunction relative to $\iota$, we can replace the above pullback by the following pullback  
\[
\begin{tikzcd}
X \arrow[r] \arrow[d]  \arrow[dr, phantom, "\lrcorner", very near start]
&\C \circ (D^n) \arrow[d] \\
\mathrm{B}_{\iota} A \arrow[r]
&\C \circ (S^n)~,
\end{tikzcd}
\]

in the category of dg $\C$-coalgebras. We conclude by the fact that the map $X \longrightarrow \mathrm{B}_{\iota} A$ is naturally isomorphic to a map $\mathrm{B}_{\iota} B \longrightarrow \mathrm{B}_{\iota} A$ where $B$ is the desired square-zero extension.
\end{proof}

\begin{remark}
The original statement concerning Artinian $\mathbb{E}_k$-algebras can be found in \cite[Theorem 7.4.1.26]{HigherAlgebra}.
\end{remark}

\begin{corollary}\label{Prop: Artinians have finite dimensional homology}
Let $A$ be an Artinian $\Omega \C$-algebra. The homology of $A$ is degree-wise finite dimensional and concentrated in finitely many non-negative homological degrees.
\end{corollary}

\begin{proof}
It follows from Proposition \ref{Prop: Artinian = Strictly Artinian}, since the homology of a strictly Artinian algebra is necessarily a perfect dg module concentrated in non-positive degrees.
\end{proof}

\begin{remark}
The $\infty$-category of Artinian $\Omega \C$-algebras is a full sub-$\infty$-category of the $\infty$-category of dg $\Omega \C$-algebras with degree-wise finite dimensional homology.
\end{remark}

\subsection{The coalgebraic point of view of formal moduli problems} The goal of this section is to define formal moduli problems exclusively in terms of coalgebras. When one considers formal moduli problems over $\mathbb{E}_\infty$-algebras, this point of view is very close to the one expressed in \cite{Hinich}: indeed, it shows that $\mathbb{E}_\infty$-coalgebras can be thought as \textit{derived formal affine stacks}. However, here we use \textit{all} coalgebras, which are not necessarily conilpotent; it agrees with the definitions in \cite{takeushi} in the non-derived case. From this point of view, formal moduli problems are defined as functors that assign to some "elementary" derived formal stacks a space, together with a "gluing condition". 

\medskip

The dg operad $\Omega\C$ is cofibrant, thus it is coadmissible. This means that the category of dg $\Omega\C$-coalgebras admits a \textit{left-transferred model structure} from dg modules, where weak-equivalences are given by quasi-isomorphisms and where cofibrations are given by degree-wise monomorphisms. 

\medskip

Let $\eta: \Omega\C \longrightarrow \III$ be the augmentation morphisms. It induces a Quillen adjunction  
\[
\begin{tikzcd}[column sep=7pc,row sep=3pc]
\mathsf{dg}~\Omega\C\text{-}\mathsf{coalg} \arrow[r, shift left=1.1ex, "\mathrm{prim}"{name=F}] 
&\mathsf{dg}~\mathsf{mod} ~, \arrow[l, shift left=.75ex, "\mathrm{triv}"{name=U}] \arrow[phantom, from=F, to=U, , "\dashv" rotate=90]
\end{tikzcd}
\]

where the functor $\mathrm{prim}(-)$ is given by primitive elements of a dg $\Omega\C$-coalgebra and the functor $\mathrm{triv}(-)$ endows any dg module with the structure of a trivial $\Omega\C$-coalgebra structure. In particular, the dg modules $S^n,D^n$ can be endowed with a trivial $\Omega\C$-coalgebra structure for all $n$ in $\mathbb{Z}$.

\begin{definition}[coArtinian $\Omega \C$-coalgebras]
The $\infty$-category of \textit{coArtinian} $\Omega\C$\textit{-coalgebras} is the smallest full subcategory of the $\infty$-category of dg $\Omega\C$-coalgebras such that  

\medskip

\begin{enumerate}
\item it contains $S^n$ for any $n \leq 0$,

\medskip

\item for any coArtinian $\Omega\C$-coalgebra $C$ and any morphism $S^n \longrightarrow C$, where $n \leq -1$, the pushout $C \amalg_{S^n} 0$ is also coArtinian.
\end{enumerate}
\end{definition}

\begin{remark}
The $\infty$-category of coArtinian coalgebras is obtained by iteratively gluing spheres (viewed as a formal derived space) endowed with the trivial coalgebra structure.
\end{remark}

\begin{definition}[Formal moduli problems over $\Omega \C$-coalgebras]
The $\infty$-category of \textit{formal moduli problem over} $\Omega\C$-coalgebras is defined as the $\infty$-category of functors 
\[
F: \mathsf{coArt}~\Omega\C\text{-}\mathsf{coalg} \longrightarrow \mathcal{S}
\]

from the $\infty$-category of coArtinian $\Omega\C$-coalgebras to the $\infty$-category of spaces satisfying the following conditions:

\begin{enumerate}
\item $F(0) \simeq \{*\}$, where $0$ is the zero $\Omega\C$-coalgebra.

\medskip

\item $F$ sends pushouts diagrams of the form
\[
\begin{tikzcd}
S^n \arrow[r] \arrow[d] \arrow[dr, phantom, "\ulcorner", very near end]
&0 \arrow[d] \\
C \arrow[r]
&C' ~,
\end{tikzcd}
\]

to pullbacks in spaces, where $n \leq -1$.
\end{enumerate}

We denote this $\infty$-category by $\mathsf{FMP}_{\Omega\C}^c$.
\end{definition}

Our goal is to show that the above $\infty$-category is equivalent the $\infty$-category of formal moduli problems over $\Omega\C$-algebras.

\begin{lemma}\label{Prop: coArtinian have finite dimensional homology}
Let $C$ be an coArtinian $\Omega\C$-coalgebra. The homology of $C$ is degree-wise finite dimensional and concentrated in finitely many non-positive homological degrees.
\end{lemma}

\begin{proof}
Any coArtinian $\Omega\C$-coalgebra is quasi-isomorphic to a \textit{strictly coArtinian} $\Omega\C$-coalgebra, that is, to a $\Omega\C$-coalgebra $C$ whose initial morphism $0 \longrightarrow C$ can be factored as 
\[
0 = C^{(n)} \longrightarrow \cdots \longrightarrow C^{(i)} \longrightarrow C^{(i-1)} \longrightarrow \cdots \longrightarrow C^{(0)} = C~,
\]

where each morphism $C^{(i)} \longrightarrow C^{(i-1)}$ fits into a homotopy pushout of the form 
\[
\begin{tikzcd}
C^{(i)} \arrow[r] \arrow[d] \arrow[dr, phantom, "\ulcorner", very near end]
&C^{(i-1)} \arrow[d] \\
0\arrow[r]
&S^n ~,
\end{tikzcd}
\]

for some $n \leq 0$. This can be shown using the same arguments as in Proposition \ref{Prop: Artinian = Strictly Artinian}, where one considers this time the complete bar-cobar adjunction relative to $\iota$
\[
\begin{tikzcd}[column sep=5pc,row sep=3pc]
          \mathsf{dg}~\Omega \C\text{-}\mathsf{coalg} \arrow[r, shift left=1.1ex, "\widehat{\Omega}_{\iota}"{name=F}] & \mathsf{dg}~\C\text{-}\mathsf{alg}, \arrow[l, shift left=.75ex, "\widehat{\mathrm{B}}_{\iota}"{name=U}]
            \arrow[phantom, from=F, to=U, , "\dashv" rotate=-90]
\end{tikzcd}
\]

which is again a Quillen equivalence when one considers the transferred structure along this adjunction on dg $\C$-algebras. Finally, it is straightforward to see that any strictly coArtinian $\Omega\C$-coalgebra satisfies the desired property. 
\end{proof}

\begin{theorem}\label{thm: coalgebraic formal moduli problems}
There is a canonical equivalence of $\infty$-categories 
\[
\mathsf{FMP}_{\Omega\C} \simeq \mathsf{FMP}_{\Omega\C}^c
\]

between the $\infty$-category of formal moduli problems over $\Omega\C$-algebras and the $\infty$-category of formal moduli problems over $\Omega\C$-coalgebras.
\end{theorem}

\begin{proof}
We showed in \cite[Section 8]{premierpapier} that the derived Sweedler dual-linear dual adjunction induces an equivalence of $\infty$-categories
\[
\begin{tikzcd}[column sep=5pc,row sep=2.5pc]
\left(\mathsf{dg}~\Omega\C \text{-}\mathsf{alg}^{\mathbf{f.d}}_{-}~[\mathrm{Q.iso}^{-1}]\right)^{\mathsf{op}} \arrow[r, shift left=1.5ex, "\mathbb{R}(-)^\circ"{name=A}]
&\mathsf{dg}~\Omega\C\text{-}\mathsf{coalg}^{\mathbf{f.d}}_{+}~[\mathrm{Q.iso}^{-1}]~. \arrow[l, shift left=.75ex, "(-)^*"{name=C}] \arrow[phantom, from=A, to=C, , "\dashv" rotate=-90]
\end{tikzcd}
\]

between the $\infty$-category of dg $\Omega\C$-algebras with bounded-above degree-wise finite dimensional homology and the $\infty$-category of $\Omega\C$-coalgebras with bounded-below degree-wise finite dimensional homology.

\medskip

It is straightforward to show that the above equivalence restricts to an equivalence of $\infty$-categories between Artinian $\Omega\C$-algebras and coArtinian $\Omega\C$-coalgebras, using Corollary \ref{Prop: Artinians have finite dimensional homology} and Lemma \ref{Prop: coArtinian have finite dimensional homology}. The equivalence between the two types of formal moduli problems follows by identifying the conditions imposed on these functor categories via the aforementioned equivalence.
\end{proof}

\subsection{Description in terms of cellular algebras over a cooperad} We identify via Koszul duality coArtinian $\Omega\C$-coalgebras with cellular $\C$-algebras, giving another description of formal moduli problems over $\Omega\C$-algebras. 

\medskip

Recall that the twisting morphism $\iota: \C \longrightarrow \Omega \C$ induces a complete bar-cobar adjunction
\[
\begin{tikzcd}[column sep=7pc,row sep=3pc]
\mathsf{dg}~\Omega\C\text{-}\mathsf{coalg} \arrow[r, shift left=1.1ex, "\widehat{\Omega}_\iota"{name=F}] 
&\mathsf{dg}~\C\text{-}\mathsf{alg}^{\mathsf{qp}\text{-}\mathsf{comp}}  ~, \arrow[l, shift left=.75ex, "\widehat{\mathrm{B}}_\iota"{name=U}] \arrow[phantom, from=F, to=U, , "\dashv" rotate=-90]
\end{tikzcd}
\]

which is a Quillen equivalence when one considers the transferred model category structure along this adjunction on the category of qp-complete dg $\C$-algebras. Over a positive characteristic field, we need to consider qp-complete algebras, which is a finer notion of completeness, in order to promote the complete bar-cobar adjunction into an equivalence. See \cite[Section 3.6]{premierpapier} for more details. Let us make this transferred model category more explicit:

\begin{enumerate}
\item weak-equivalences are given by morphisms $f$ such that $\widehat{\mathrm{B}}_\iota(f)$ is a quasi-isomorphisms,
\item fibrations are given by degree-wise epimorphisms,
\item cofibrations are determined by the left-lifting property.
\end{enumerate}

The $\infty$-category of $\C$-algebras that we consider is given by localizing this model category at the aforementioned class of weak-equivalences, which is a sub-class of quasi-isomorphisms.

\begin{proposition}\label{prop: W.eq are quasi-isos}
The class of weak-equivalences of dg $\C$-algebras is strictly included in the class of quasi-isomorphisms of dg $\C$-algebras.
\end{proposition}

\begin{proof}
Follows from \cite[Section 7]{premierpapier}.
\end{proof}

\begin{definition}[Cellular $\C$-algebras]
The $\infty$-category of \textit{cellular} $\C$-algebras is the smallest full sub-category of the $\infty$-category of $\C$-algebras 

\medskip

\begin{enumerate}
\item it contains $(S^n)^{\C}$ for any $n \leq 0$,

\medskip

\item for any cellular $\C$-algebra $\mathfrak{g}$ and any morphism $(S^n)^{\C} \longrightarrow \mathfrak{g}$, where $n \leq -1$, the pushout $\mathfrak{g} \amalg_{(S^n)^{\C}} 0^\C$ is also cellular.
\end{enumerate}
\end{definition}

\begin{remark}
    When $\C$ is $0$-reduced, meaning $\C(0) = 0$, the initial $\C$-algebra $0^\C$ is simply $0$.
\end{remark}

\begin{remark}
Following the terminology of \cite{presentationalgebras}, these objects should be referred to as $0$-cellular objects.
\end{remark}

\begin{definition}[Cocartesian presheaves]
A \textit{cocartesian presheaf} $F$ on the $\infty$-category of cellular $\C$-algebras amounts to the data of a functor
    \[
    F: \mathsf{Cell}~\C\text{-}\mathsf{alg}^{\mathsf{op}} \longrightarrow \mathcal{S},
    \]
    satisfying the following conditions:

\begin{enumerate}
\item $F(0^\C) \simeq \{*\}$, where $0^\C$ is the initial $\C$-algebra.

\medskip

\item $F$ sends pushouts diagrams of the form
\[
\begin{tikzcd}
(S^n)^{\C} \arrow[r] \arrow[d] \arrow[dr, phantom, "\ulcorner", very near end]
&0^\C \arrow[d] \\
\mathfrak{g} \arrow[r]
&\mathfrak{g}'~,
\end{tikzcd}
\]

to pullbacks in spaces, where $n \leq -1$.
\end{enumerate}

\medskip

We denote $\Psh^\lrcorner (\mathsf{Cell}~\C\text{-}\mathsf{alg})$ the full reflective sub $\infty$-category of presheaves on cellular $\C$-algebras.
\end{definition}

\begin{proposition}\label{prop: FMP simeq Functors from cellular C-algebras}
The $\infty$-category of formal moduli problem over $\Omega\C$-algebras is canonically equivalent to the $\infty$-category of cocartesian presheaves on cellular $\C$-algebras.
\end{proposition}

\begin{proof}
Follows directly from Theorem \ref{thm: coalgebraic formal moduli problems}, using the fact that the complete bar-cobar adjunction with respect to $\iota$ is a Quillen equivalence, simply by noticing that the image of $S^n$ under the complete cobar functor is given by $(S^n)^{\C}$.
\end{proof}

\begin{remark}
One can not apply the methods of \cite{presentationalgebras} in order to conclude that formal moduli problems over $\Omega \C$-algebras are always encoded by qp-complete dg $\C$-algebras. The obstruction is that the monad 
\[
(-)^{\C} \coloneqq \displaystyle \prod_{n \geq 0} \mathrm{Hom}_{\mathbb{S}_n}(\C(n), (-)^{\otimes n})
\]

clearly does not preserve sifted homotopy colimits. Therefore one can not conclude that the category of qp-complete dg $\C$-algebras admits a presentation in terms of cellular $\C$-algebras. This fact is what forces us to consider instead the category of dg $\C^*$-algebras, which admits such a presentation.
\end{remark}

\section{Divided power algebras}
Our goal in this section is to define the $\infty$-category of $\C^*$-algebras, which, under certain hypothesis, will encode formal moduli problems over $\Omega\C$-algebras. Since we are working over a field $\kk$ of any characteristic, we first give a series of recollections on tame model structures and semi-model structures on $\mathbb{S}$-tame operads, which mostly come from \cite{pdalgebras}. 

\subsection{The tame model structure}
Let $R$ be a dg unital associative algebra. It this section we briefly recall the \textit{tame} or \textit{contraderived} model structure on the category of dg $R$-modules. Let us define quasi-free dg $R$-modules to be those whose underlying graded $R$-module is free and let us define graded-projective dg $R$-modules to be retracts of quasi-free dg $R$-modules. The tame model structure is given by the following classes of maps:

\begin{enumerate}
\item the class of fibrations is given by degree-wise epimorphisms,

\item the class of cofibrations is given by monomorphisms whose cokernel is a graded-projective dg $R$-module,

\item and the class of weak-equivalences is given by morphisms $f: A \longrightarrow B$ which induce quasi-isomorphisms 
\[
f_*: \mathrm{hom}(P,A) \longrightarrow \mathrm{hom}(P,B)~,
\]

for all graded-projective dg $R$-modules $P$, that is, for all dg $R$-modules such that $P_k$ is a projective $R$-module for all $k$ in $\mathbb{Z}$.
\end{enumerate}

See \cite{TwoKinds} where contraderived categories are introduced. The projective model structure can be obtained as a right Bousfield localization of the tame model structure. Both model structures actually agree on bounded below dg $R$-modules, meaning that any quasi-isomorphism of bounded below dg $R$-modules is a tame equivalence.

\begin{example}\label{Ex: qi = tame eq}
When $R$ is a field $\kk$, the tame model structure on dg $\kk$-modules coincides with the projective model structure.
\end{example}

\begin{remark}
When $R$ is a \textit{commutative} dga which \textit{homologically bounded} and \textit{coherent}, the $\infty$-category presented by this model structure can be identified with the $\infty$-category of \textit{pro-coherent} $R$-modules. See \cite[Section 8]{Liealgebroids}. 
\end{remark}

In particular, one can consider the tame model structure on dg $\kk[\mathbb{S}_n]$-modules for any $n \geq 0$. This in turn induces a tame model structure on the category of dg $\mathbb{S}$-modules, where weak-equivalences are tame equivalences arity-wise. 

\begin{definition}[$\mathbb{S}$-tame]
A dg $\mathbb{S}$-module $M$ is said to be $\mathbb{S}$-\textit{tame} if $M_n$ is a cofibrant object in the tame model structure of dg $\mathbb{S}_n$-modules for all $n$. In other words, for every natural integer $n$, $M(n)$ is a retract of a quasi-free right $\mathbb S_n$-modules.
\end{definition}

\begin{example}
A quasi-free dg $\mathbb{S}$-module, although it might not be $\mathbb{S}$-projective nor $\mathbb{S}$-injective, it is always $\mathbb{S}$-tame.
\end{example}

\begin{proposition}[{\cite[Proposition 4.31]{pdalgebras}}]\label{prop: tame model structures on P algebras}
Let $\mathcal{P}$ be a dg operad whose underlying dg $\mathbb{S}$-module is $\mathbb{S}$-tame. The category of dg $\mathcal{P}$-algebras admits a \textit{cofibrantly generated semi-model category structure} given by 

\begin{enumerate}
\item weak-equivalences are given by quasi-isomorphisms,
\item fibrations are given by degree-wise epimorphisms,
\item cofibrations are determined by the left-lifting property.
\end{enumerate}
\end{proposition}

\begin{remark}
The $\infty$-category presented by this model structure does not depend on the choice of the dg operad $\mathcal{P}$ in the following sense: any tame equivalence of dg operads induces a Quillen equivalence between the corresponding semi-model category structures.
\end{remark}

L. Brantner, R. Campos and J. Nuiten work in a more general setting where $\kk$ is any coherent ring, therefore the semi-model structure on dg $\mathcal{P}$-algebras is also given by tame equivalences. The same phenomenon as in Example \ref{Ex: qi = tame eq} reduces tame equivalences of dg $\mathcal{P}$-algebras to quasi-isomorphisms when the base category is that of dg modules over a field. We refer to \cite[Section 4]{pdalgebras} for more details. 

\subsection{The linear dual operad}
For the rest of this section, we fix a quasi-planar conilpotent dg cooperad $\C$.

\begin{lemma}
The dg $\mathbb{S}_n$-module $\C(n)^*$ is $\mathbb{S}_n$-injective module.
\end{lemma}

\begin{proof}
The dg $\mathbb{S}_n$-module $\C(n)$ is projective since $\C$ is quasi-planar, see \cite[Section 2]{premierpapier}. The linear dual of a $\mathbb{S}_n$-projective module is a $\mathbb{S}_n$-injective module, see \cite[Section 1]{premierpapier} for more details.
\end{proof}

\begin{lemma}
There is a canonical dg operad structure on the dg $\mathbb{S}$-module $\{\C(n)^*\}$.
\end{lemma}

\begin{proof}
The norm map from coinvariants to invariants of symmetric groups yields a morphism 
\[
\mathrm{norm}: M \comp N \longrightarrow M~ \tilde{\comp}~ N
\]
of $\mathbb{S}$-modules. It makes the the identity of symmetric sequences a oplax monoidal functor; see \cite{FresseDP}. Thus $\C$ is in particular a comonoid with respect to the composition product $\tilde{\circ}$. This structure endows $\C$ with a unital partial cooperad structure, that is, a coalgebra over the coloured operad that encodes operads, see \cite[Section 1.2]{lucio2022curved}. Hence, its linear dual is endowed with an operad structure.
\end{proof}

\begin{remark}
The norm map 
\[
\mathrm{norm}: \C \comp \C \cong \C~ \tilde{\comp}~ \C
\]
is in fact a isomorphism in this case, as the action of the symmetric groups on the $\mathbb{S}$-module $\C$ is degree-wise free. 
\end{remark}

\subsection{Divided power algebras} The dg operad $\C^*$ induces a $1$-categorical monad on the category of dg module via its Schur functor: 
\[
\displaystyle \bigoplus_{n \geq 0} \C^*(n) \otimes_{\mathbb{S}_n} (-)^{\otimes n}: \mathsf{dg}~\mathsf{mod} \longrightarrow \mathsf{dg}~\mathsf{mod}~.
\]

\begin{lemma}
The canonical natural transformation 
\[
\displaystyle \C^* \circ (-) = \bigoplus_{n \geq 0} \C^*(n) \otimes_{\mathbb{S}_n} (-)^{\otimes n} \cong \bigoplus_{n \geq 0} \left(\C^*(n) \otimes (-)^{\otimes n}\right)^{\mathbb{S}_n} = \C^* ~\tilde{\circ}~(-)
\]
induced by the norm map with respect to the actions of the symmetric groups is an isomorphism.
\end{lemma}

\begin{proof}
The dg $\mathbb{S}_n$-module $\C^*(n)$ is quasi-free, thus the norm map is an isomorphism.
\end{proof}

The linear dual $\C^*$ has both the structure of a dg operad and of a \textit{divided power} dg operad, that is, a monoid for the composition product $\tilde{\circ}$. Hence one can consider divided power $\C^*$-algebras, that is, chain  complexes equipped with the structure of an algebra over the monad $\C^* ~\tilde{\comp}~(-)$. However, since the action of the symmetric groups on $\C^*$ is degree-wise free, the categories of dg $\C^*$-algebras and of divided power $\C^*$-algebras are isomorphic, so no divided power operations can appear at the $1$-categorical level.

\medskip

At the $\infty$-categorical level, this equivalence between usual $\C^*$-algebras and divided power $\C^*$-algebras no more holds. Since $\C^*(n)$ is \textit{not} $\mathbb{S}_n$-projective, we \textit{don't} have a natural quasi-isomorphism of dg modules 
\[
\displaystyle \bigoplus_{n \geq 0} \C^*(n) \otimes_{\mathbb{S}_n} (-)^{\otimes n} \not\simeq \bigoplus_{n \geq 0} \left(\C^*(n) \otimes (-)^{\otimes n}\right)_{h\mathbb{S}_n}~,
\]

where on the right we consider homotopy coinvariants. This illustrates the fact that although divided power operations do not appear $1$-categorically on a dg $\C^*$-algebra, the $\infty$-category of $\C^*$-algebras will still behaves like an $\infty$-category of divided power algebras and not like an $\infty$-category of usual algebras. 

\medskip

The $\infty$-category of divided power $\C^*$-algebras can be presented by a semi-model structure.

\begin{proposition}
The category of dg $\C^*$-algebras admits a \textit{cofibrantly generated semi-model category structure} given by 

\begin{enumerate}
\item weak-equivalences are given by quasi-isomorphisms,
\item fibrations are given by degree-wise epimorphisms,
\item cofibrations are determined by the left-lifting property.
\end{enumerate}
\end{proposition}

\begin{proof}
The underlying dg $\mathbb S$-module of $\C^*$ is quasi-free and therefore $\mathbb{S}$-tame. So, the result follows from Proposition \ref{prop: tame model structures on P algebras}.
\end{proof}

Since this semi-model category structure is transferred from the category of dg modules, the free-forget adjunction induces a monadic adjunction at the level of $\infty$-categories.

\begin{proposition}[{\cite[Section 4]{pdalgebras}}]\label{prop: monadicity of the dervied adjunction}
The adjunction induced at the level of $\infty$-categories by the free-forgetful Quillen adjunction
\[
\begin{tikzcd}[column sep=7pc,row sep=3pc]
\mathsf{dg}~\mathsf{mod}~[\mathrm{Q.iso}^{-1}] \arrow[r, shift left=1.1ex, "\C^* \circ ~ (-)"{name=F}] 
&\mathsf{dg}~\C^*\text{-}\mathsf{alg}~[\mathrm{Q.iso}^{-1}] ~, \arrow[l, shift left=.75ex, "\mathrm{U}"{name=U}] \arrow[phantom, from=F, to=U, , "\dashv" rotate=-90]
\end{tikzcd}
\]
is monadic and the forgetful functor preserves sifted colimits. 
\end{proposition}

\begin{proof}
The right adjoint $\infty$-functor preserves sifted colimits by \cite[Proposition 4.32]{pdalgebras} combined with \cite[Lemma 4.26, Remark 4.30]{pdalgebras}. Moreover, it is clearly conservative. Thus, it is also monadic by the monadicity theorem for $\infty$-categories \cite[Theorem 4.7.0.3]{HigherAlgebra}.
\end{proof}

\section{Cellular algebras and formal moduli problems}
The goal of this section is to construct the $\infty$-category of cellular $\C^*$-algebras, and show that the $\infty$-category of $\C^*$-algebras admits a presentation in terms of cocartesian presheaves on cellular $\C^*$-algebras. This will allow to reduce any formal moduli problem type of statement to a comparison result between cellular $\C$-algebras and cellular $\C^*$-algebras. That is, between the \textit{absolute} and the \textit{non-absolute} versions of the same algebraic structure. For the rest of this section, $\C$ is a quasi-planar conilpotent dg cooperad.

\subsection{A presentation of algebras over a monad}
We use the main theorem of \cite{presentationalgebras} in order to give a presentation of the $\infty$-category of $\C^*$-algebras purely in terms of cocartesian presheaves on cellular $\C^*$-algebras. Let
\[
\begin{tikzcd}[column sep=7pc,row sep=3pc]
\mathsf{dg}~\mathsf{mod}~[\mathrm{Q.iso}^{-1}] \arrow[r, shift left=1.1ex, "\mathrm{T}(-)"{name=F}] 
&\mathrm{T}\text{-}\mathsf{alg} ~, \arrow[l, shift left=.75ex, "\mathrm{U}"{name=U}] \arrow[phantom, from=F, to=U, , "\dashv" rotate=-90]
\end{tikzcd}
\]
be a monadic adjunction of $\infty$-categories.

\begin{definition}[$m$-cellular $\mathrm{T}$-algebras]
Let $m$ be in $\mathbb{Z} \cup \{\infty\}$. The $\infty$-category of $m$-\textit{cellular} $\mathrm{T}$-algebras is the smallest full sub-category of $\mathrm{T}$-algebras
 
\begin{enumerate}
\item it contains $\mathrm{T}(S^n)$ for any $n \leq m$,

\medskip

\item for any $m$-cellular $\mathrm{T}$-algebra $A$ and any morphism $\mathrm{T}(S^n) \longrightarrow A$, where $n \leq m-1$, the pushout $A \amalg_{\mathrm{T}(S^n)} \mathrm{T}(0)$ is also cellular.
\end{enumerate}
\end{definition}

\begin{definition}[Cocartesian presheaves]
A \textit{cocartesian presheaf} $F$ on the $\infty$-category on $m$-cellular $\mathrm{T}$-algebras amounts to the data of a functor
    \[
    F: \mathsf{Cell}^{\leq m}~\mathrm{T}\text{-}\mathsf{alg}^{\mathsf{op}} \longrightarrow \mathcal{S},
    \]
    satisfying the following conditions:

\begin{enumerate}
\item $F(\mathrm{T}(0)) \simeq \{*\}$, where $\mathrm{T}(0)$ is the initial $\mathrm{T}$-algebra,

\medskip

\item $F$ sends pushouts diagrams of the form
\[
\begin{tikzcd}
\mathrm{T}(S^n) \arrow[r] \arrow[d] \arrow[dr, phantom, "\ulcorner", very near end]
&\mathrm{T}(0) \arrow[d] \\
A \arrow[r]
&A'~,
\end{tikzcd}
\]

to pullbacks in spaces, where $n \leq m-1$.
\end{enumerate}

We denote $\Psh^\lrcorner\left(\mathsf{Cell}^{\leq m}~\mathrm{T}\text{-}\mathsf{alg}\right)$ the full reflective sub $\infty$-category of presheaves on cellular $\mathrm{T}$-algebras.
\end{definition}

\begin{theorem}[{\cite[Theorem 3]{presentationalgebras}}]\label{thm: article Brice présentation cellulaire}
Let $m$ be in $\mathbb{Z} \cup \{\infty\}$ and let
\[
\begin{tikzcd}[column sep=7pc,row sep=3pc]
\mathsf{dg}~\mathsf{mod}~[\mathrm{Q.iso}^{-1}] \arrow[r, shift left=1.1ex, "\mathrm{T}(-)"{name=F}] 
&\mathrm{T}\text{-}\mathsf{alg} ~, \arrow[l, shift left=.75ex, "\mathrm{U}"{name=U}] \arrow[phantom, from=F, to=U, , "\dashv" rotate=-90]
\end{tikzcd}
\]
be a monadic adjunction of $\infty$-categories such that the forgetful functor $\mathrm{U}$ preserves sifted colimits. There is a canonical equivalence of $\infty$-categories
\[
\begin{tikzcd}[column sep=5pc,row sep=5pc]
\Psh^\lrcorner\left(\mathsf{Cell}^{\leq m}~\mathrm{T}\text{-}\mathsf{alg}\right) \arrow[r,"\mathrm{Lan}(i_{\mathrm{T}})"{name=LDC},shift left=1.1ex ] 
&\mathrm{T}\text{-}\mathsf{alg}~,\arrow[l,"\mathrm{Y}"{name=TD},shift left=1.1ex] \arrow[phantom, from=TD, to=LDC, , "\dashv" rotate=-90] 
\end{tikzcd}
\] 
where $\mathrm{Y}$ is the Yoneda nerve $A \mapsto \mathrm{Map}(-, A)$ and $\mathrm{Lan}(i)$ is the left Kan extension of the inclusion functor $i_{\mathrm{T}}$ from cellular $\mathrm{T}$-algebras into all $\mathrm{T}$-algebras.
\end{theorem}

\begin{notation}
The $\infty$-category of $0$-cellular $\C^*$-algebras will be referred simply as the $\infty$-category of cellular $\C^*$-algebras.
\end{notation}

\begin{theorem}\label{thm: présentation des C^*-algebres}
There is a canonical equivalence of $\infty$-categories
\[
\begin{tikzcd}[column sep=5pc,row sep=5pc]
\Psh^\lrcorner(\mathsf{Cell}~\C^*\text{-}\mathsf{alg}) \arrow[r,"\mathrm{Lan}(i_{\C^*})"{name=LDC},shift left=1.1ex ] 
&\mathsf{dg}~ \C^*\text{-}\mathsf{alg}~[\mathrm{Q.iso}^{-1}]~,\arrow[l,"\mathrm{Y}"{name=TD},shift left=1.1ex] \arrow[phantom, from=TD, to=LDC, , "\dashv" rotate=-90] 
\end{tikzcd}
\] 
between the $\infty$-category of $\C^*$-algebras and the $\infty$-category of cocartesian presheaves on cellular $\C^*$-algebras.
\end{theorem}

\begin{proof}
We apply Theorem \ref{thm: article Brice présentation cellulaire} with $m = 0$ to monad $\C^* \circ (-)$, which by Propositions \ref{prop: monadicity of the dervied adjunction} induces a monad which preserves sifted colimits. 
\end{proof}

Every cellular $\C^*$-algebra is quasi-isomorphic strictly cellular dg $\C^*$-algebra, which is cofibrant in the semi-model category of dg $\C^*$-algebras.

\begin{definition}[Strictly cellular dg $\C^*$-algebra]
A \textit{strictly cellular} dg $\C^*$-algebra is an quasi-free dg $\C^*$-algebra of the form
\[
\left(\C^* \circ V, d_{~\C^* \circ V} + d_2 \right)~,
\]
where 
\begin{enumerate}
\item $V$ is a dg module concentrated in degrees $\leq  0$ of total finite dimension, endowed with a exhaustive filtration 
\[
0 = V(-1) \subseteq V(0) \subseteq V(1) \subseteq \cdots \subseteq V~,
\]
\item the differential $d_2$ satisfies that
\[
d_2(V(i)) \subseteq \C^* \circ V(i-1)~.
\]
\end{enumerate}
\end{definition}

\begin{proposition}\label{prop: charactérisation des C*-cellulaires}
Any cellular $\C^*$-algebra is quasi-isomorphic to a strictly cellular dg $\C^*$-algebra.
\end{proposition}

\begin{proof}
Follows from the definition of a cellular $\C^*$-algebra. Indeed, they are given by a finite number of (homotopy) pushouts of the form
\[
\begin{tikzcd}[column sep=3.5pc,row sep=3.5pc]
\C^* \circ S^n \arrow[r] \arrow[d] \arrow[dr, phantom, "\ulcorner", very near end]
&\C^* \circ D^n \arrow[d] \\
\C^* \circ V(i-1) \arrow[r]
&\C^* \circ V(i)~,
\end{tikzcd}
\]
so the generators are a finite dimensional dg module in degrees $\leq  0$, and it clearly follows that $d_2(V(i)) \subseteq \C^* \circ V(i-1)~.$ 
\end{proof}

%--------------------------------------------------------

\subsection{The canonical adjunction}
There is a canonical adjunction between the $\infty$-category of $\C$-algebras and the $\infty$-category of $\C^*$-algebras. This adjunction induces an adjunction between the $\infty$-categories of formal moduli problems over $\Omega \C$-algebras and of $\C^*$-algebras.

\begin{lemma}\label{lemma: adjunction Ab - Res}
There is a Quillen adjunction 
\[
\begin{tikzcd}[column sep=5pc,row sep=5pc]
\mathsf{dg}~\mathcal{C}\text{-}\mathsf{alg}^{\mathsf{qp}\text{-}\mathsf{comp}} \arrow[r,"\mathrm{Res}"{name=LDC},shift left=1.1ex ] 
&\mathsf{dg}~\mathcal{C}^*\text{-}\mathsf{alg}~,\arrow[l,"\mathrm{Ab}"{name=TD},shift left=1.1ex] \arrow[phantom, from=TD, to=LDC, , "\dashv" rotate=90] 
\end{tikzcd}
\] 

between the category of qp-complete dg $\C$-algebras and the category of dg $\C^*$-algebras.
\end{lemma}

\begin{proof}
There is a morphism of monads

\[
\bigoplus_{n \geq 0} \mathcal{C}(n)^* \otimes_{\mathbb{S}_n} (-)^{\otimes n} \rightarrowtail \prod_{n \geq 0} \mathrm{Hom}_{\mathbb{S}_n} (\mathcal{C}(n), (-)^{\otimes n}) 
\]

which induces the adjunction $\mathrm{Ab} \dashv \mathrm{Res}$, given by the composition of the norm map from coinvariants to invariants with the natural inclusion $\mathcal{C}(n)^* \otimes (-)^{\otimes n} \longrightarrow \mathrm{Hom} (\mathcal{C}(n), (-)^{\otimes n})$. The functor $\mathrm{Res}$ is a right Quillen functor since it preserves fibrations and since any weak-equivalence of dg $\C$-algebras is in particular a quasi-isomorphism by Proposition \ref{prop: W.eq are quasi-isos}.
\end{proof}

\begin{lemma}\label{lemma: foncteur Ab infini-cat}
The derived functor $\mathbb{L}\mathrm{Ab}$ restricts to a functor
\[
\begin{tikzcd}[column sep=4pc,row sep=5pc]
\mathsf{Cell}~\C^*\text{-}\mathsf{alg} \arrow[r,"\mathrm{Ab}"{name=TD}]
&\mathsf{Cell}~\C\text{-}\mathsf{alg}~,
\end{tikzcd}
\] 

between the sub-$\infty$-categories of cellular $\C^*$-algebras and cellular $\C$-algebras.
\end{lemma}

\begin{proof}
The adjunction of Lemma \ref{lemma: adjunction Ab - Res} induces an adjunction between the $\infty$-categories of $\C$-algebras and of $\C^*$-algebras. Clearly, for any $n \leq 0$, there is a weak-equivalence (in fact, an isomorphism) of $\C$-algebras
\[
\mathrm{Ab}(\C^* \circ S^n)~ \simeq ~ (S^n)^{\C}~,
\]
and moreover, it preserves homotopy pushouts, therefore it sends any cellular $\C^*$-algebra to a cellular $\C$-algebra.
\end{proof}

\begin{remark}
The functor $\mathrm{Ab}$ preserves (quasi)-free objects in general.
\end{remark}

Since the left adjoint functor $\mathbb{L}\mathrm{Ab}$ restricts to cellular objects on both sides, it induces the following commutative square of left adjoint functors
\[
\begin{tikzcd}[column sep=2pc,row sep=3pc]
    \Psh(\mathsf{Cell}~\C^*\text{-}\mathsf{alg})
    \ar[r,"\mathbb{L}\mathrm{Ab}_!"] \ar[d,"\mathrm{Lan}(i_{\C^*})",swap]
    & \Psh(\mathsf{Cell}~\C\text{-}\mathsf{alg})
    \ar[d,"\mathrm{Lan}(i_{\C})"]
    \\
    \mathsf{dg}~\mathcal{C}^*\text{-}\mathsf{alg}~[\mathrm{Q.iso}^{-1}]
    \ar[r,"\mathbb{L}\mathrm{Ab}"]
    &\mathsf{dg}~\mathcal{C}\text{-}\mathsf{alg}^{\mathsf{qp}\text{-}\mathsf{comp}}~[\mathrm{W}^{-1}]~.
\end{tikzcd}
\]
The left Kan extensions $\mathrm{Lan}(i)$ of the inclusion of cellular algebras into all algebras factor on both sides through cocartesian presheaves. This gives the following commutative square
\[
\begin{tikzcd}[column sep=2pc,row sep=3pc]
    \Psh^\lrcorner(\mathsf{Cell}~\C^*\text{-}\mathsf{alg})
    \ar[r,"\mathbb{L}\mathrm{Ab}_!"] \ar[d,"\mathrm{Lan}(i_{\C^*})",swap]
    & \Psh^\lrcorner(\mathsf{Cell}~\C\text{-}\mathsf{alg})
    \ar[d,"\mathrm{Lan}(i_{\C})"]
    \\
    \mathsf{dg}~\mathcal{C}^*\text{-}\mathsf{alg}~[\mathrm{Q.iso}^{-1}]
    \ar[r,"\mathbb{L}\mathrm{Ab}"]
    &\mathsf{dg}~\mathcal{C}\text{-}\mathsf{alg}^{\mathsf{qp}\text{-}\mathsf{comp}}~[\mathrm{W}^{-1}]~.
\end{tikzcd}
\]

of left adjoint functors. By Theorem \ref{thm: présentation des C^*-algebres}, the functor $\mathrm{Lan}(i_{\C^*})$ is an equivalence. By Proposition \ref{prop: FMP simeq Functors from cellular C-algebras}, there is also an equivalence of categories
\[
\widehat{\mathrm{B}}_{\iota,!}: \FMP^c_{\Omega \C} \longrightarrow \Psh^\lrcorner(\mathsf{Cell}~\C\text{-}\mathsf{alg})~.
\]
Therefore pre and post composing the adjunction 
\[
\begin{tikzcd}[column sep=5pc,row sep=5pc]
 \Psh^\lrcorner(\mathsf{Cell}~\C^*\text{-}\mathsf{alg})~\arrow[r,"\mathbb{L}\mathrm{Ab}_!"{name=LDC},shift left=1.1ex ] 
&\Psh^\lrcorner(\mathsf{Cell}~\C\text{-}\mathsf{alg}),\arrow[l,"\mathbb{L}\mathrm{Ab}^*"{name=TD},shift left=1.1ex] \arrow[phantom, from=TD, to=LDC, , "\dashv" rotate=-90] 
\end{tikzcd}
\] 
with the aforementioned two equivalences, we get an adjunction
\[
\begin{tikzcd}[column sep=5pc,row sep=5pc]
\mathsf{dg}~\C^*\text{-}\mathsf{alg}~[\mathrm{Q.iso}^{-1}]~\arrow[r,"\Phi"{name=LDC},shift left=1.1ex ] 
&\mathsf{FMP}_{\Omega\C}^c,\arrow[l,"\Psi"{name=TD},shift left=1.1ex] \arrow[phantom, from=TD, to=LDC, , "\dashv" rotate=90] 
\end{tikzcd}
\] 
\vspace{0.1pc}

relating the $\infty$-category of divided power dg $\C^*$-algebras to the $\infty$-category of formal moduli problems over $\Omega \C$-algebras.

\begin{lemma}\label{lemma: cellular equiv = quasi-iso somme produit sur les parfaits2}
The following assertions about the functor 
\[
\begin{tikzcd}[column sep=4pc,row sep=5pc]
\mathsf{Cell}~\C^*\text{-}\mathsf{alg}
\arrow[r,"\mathbb{L}\mathrm{Ab}"{name=TD}]
&\mathsf{Cell}~\C\text{-}\mathsf{alg}
\end{tikzcd}
\] 
are equivalent
\begin{enumerate}
    \item it is fully faithful,
    \item it is an equivalence of $\infty$-categories.
\end{enumerate}
\end{lemma}

\begin{proof}
Suppose the functor $\mathbb{L}\mathrm{Ab}$ is fully faithful, then it induces an equivalence between the $\infty$-categories of cellular $\C^*$-algebras and of cellular $\C$-algebras. Indeed, since cellular $\C$-algebras are the smallest sub-$\infty$-category of $\C$-algebras generated by $(S^n)^{\C}$ under certain pushouts, $\mathbb{L}\mathrm{Ab}$ is also essentially surjective. 
\end{proof}

\begin{theorem}\label{Thm: the canonical adjunction}
The following assertions are equivalent
\begin{enumerate}
\medskip
    \item The adjunction 
    	\[
\begin{tikzcd}[column sep=5pc,row sep=5pc]
\mathsf{FMP}_{\Omega\C} \arrow[r,"\Phi"{name=LDC},shift left=1.1ex ] 
&\mathsf{dg}~\C^*\text{-}\mathsf{alg}~[\mathrm{Q.iso}^{-1}]~,\arrow[l,"\Psi"{name=TD},shift left=1.1ex] \arrow[phantom, from=TD, to=LDC, , "\dashv" rotate=90] 
\end{tikzcd}
\] 
     is equivalence of $\infty$-categories.

\medskip

    \item The functor 
    \[
\begin{tikzcd}[column sep=4pc,row sep=5pc]
\mathsf{Cell}~\C^*\text{-}\mathsf{alg}
\arrow[r,"\mathbb{L}\mathrm{Ab}"{name=TD}]
&\mathsf{Cell}~\C\text{-}\mathsf{alg}~.
\end{tikzcd}
	\] 
	is fully faithful.

\medskip

    \item For every strictly cellular dg $\C^\ast$-algebra $A$, the unit map of the adjunction $\mathbb{L}\mathrm{Ab} \dashv \Res$ is a quasi-isomorphism.
\end{enumerate}
\end{theorem}

\begin{proof}
The implication $(3) \Rightarrow (2)$ follows from Proposition \ref{prop: charactérisation des C*-cellulaires}. The implication $(2) \Rightarrow (1)$ follows from Lemma \ref{lemma: cellular equiv = quasi-iso somme produit sur les parfaits2}.

\medskip

The implication $(1) \Rightarrow (2)$ is a consequence of the fact that cellular $\C$-algebras and cellular $\C^\ast$-algebras are full subcategories of, respectively, formal moduli problems over $\Omega\C$-algebras and $\C^*$-algebras.

\medskip

Finally, let us suppose $(2)$. Let $\mathfrak{g},\mathfrak{g}'$ be two cellular $\C^\ast$ algebra. Without any loss of generality, we can suppose that $\mathfrak{g}$ is strictly cellular by Proposition \ref{prop: charactérisation des C*-cellulaires}. The map
\[
\Map{\mathsf{dg}~\C^*\text{-}\mathsf{alg}~[\mathrm{Q.iso}^{-1}]}{\mathfrak{g}'}{\mathfrak{g}} \simeq \Map{\mathsf{dg}~\C^*\text{-}\mathsf{alg}~[\mathrm{Q.iso}^{-1}]}{\mathfrak{g}'}{ \Res ~ \Ab (\mathfrak{g})} 
\]
is a equivalence since $\mathbb{L}\mathrm{Ab}$ is fully faithful. This means that the derived unit map 
\[
\mathbb{L}\eta_\mathfrak{g}: \mathfrak{g} \longrightarrow \Res ~ \mathbb{L}\mathrm{Ab} (\mathfrak{g})
\]
is local with respect to every cellular algebra $\mathfrak{g}'$. In turn, this implies that it is a weak-equivalence, since every dg $\C^\ast$ algebra is a colimit of cellular algebras.
\end{proof}   

\subsection{Formal moduli problem associated to an algebra}
For the rest of this subsection, we will suppose that there is an equivalence between the $\infty$-category of formal moduli problems of $\Omega \C$-algebras and the $\infty$-category of dg $\C^*$-algebras. This will for instance be the case if $\C$ is \textit{tempered}, as explained in Section \ref{section: tempered cooperads}. Our goal here of this subsection is to describe $\Psi$ for different kinds of dg $\C^*$-algebras.

\begin{lemma}
The functor which associates to a dg $\C^*$-algebra the formal moduli problem it encodes is given by 
\[
\begin{tikzcd}[column sep=3.5pc,row sep=0.5pc]
\Psi: \mathsf{dg}~\C^*\text{-}\mathsf{alg}~[\mathrm{Q.iso}^{-1}] \arrow[r,]
& \Psh^\lrcorner\left(\mathsf{coArt}~\Omega \C\text{-}\mathsf{coalg}\right) \\
\mathfrak{g} \arrow[r,mapsto]
& \mathrm{Map}_{\mathsf{dg}~\C^*\text{-}\mathsf{alg}[\mathrm{Q.iso}^{-1}]}(\Res~\widehat{\Omega}_\iota (-),\mathfrak{g})~,
\end{tikzcd}
\]

at the $\infty$-categorical level.
\end{lemma}

\begin{proof}
Since $\mathbb{L}\mathrm{Ab}$ is an equivalence between cellular objects, $\Res$ preserves pushouts of cellular $\C$-algebras. So $\Psi(\mathfrak{g})$ is indeed a formal moduli problem. The result follows from direct inspection.
\end{proof}

\begin{remark}
We chose to describe the formal moduli problem $\Psi(\mathfrak{g})$ using the test category of coArtinian $\Omega \C$-coalgebras as it is more natural to do so. A description in terms of Artinian $\Omega \C$-algebras can by obtained directly by pre-composing $\Psi(\mathfrak{g})$ with the derived Sweedler dual functor $\mathbb{R}(-)^\circ$.
\end{remark}

\begin{theorem}\label{thm: rectification of fmps}
Let $\mathfrak{g}$ be a dg $\C^*$-algebra. The following assertions are equivalent. 

\medskip

\begin{enumerate}
\item The algebra $\mathfrak{g}$ is in the homotopical essential image of the functor 
\[
\begin{tikzcd}[column sep=5pc,row sep=5pc]
\mathsf{dg}~\mathcal{C}\text{-}\mathsf{alg}^{\mathsf{qp}\text{-}\mathsf{comp}}~[\mathrm{W}^{-1}] \arrow[r,"\mathrm{Res}"{name=LDC}] 
&\mathsf{dg}~\mathcal{C}^*\text{-}\mathsf{alg}~[\mathrm{Q.iso}^{-1}]~.
\end{tikzcd}
\] 

\item The formal moduli problem $\Psi(\mathfrak{g})$ is \textit{corepresentable} by a dg $\Omega \C$-coalgebra $V_\mathfrak{g}$, meaning it is given by 
\[
\begin{tikzcd}[column sep=2pc,row sep=0.5pc]
\Psi(\mathfrak{g}): \mathsf{coArt}~\Omega \C\text{-}\mathsf{coalg} \arrow[r,]
&\mathcal{S} \\
C \arrow[r,mapsto]
&\mathrm{Map}_{\mathsf{dg}~\Omega\C\text{-}\mathsf{coalg}[\mathrm{Q.iso}^{-1}]}(C,V_\mathfrak{g})~.
\end{tikzcd}
\]
\end{enumerate}
\end{theorem}

\begin{proof}
Follows from the commutativity of the following diagram of functors
\[
    \begin{tikzcd}[column sep=2pc,row sep=3pc]
    \Psh^\lrcorner(\mathsf{Cell}~\C^*\text{-}\mathsf{alg})
    \ar[r, "\mathbb{L}\mathrm{Ab}_!"]
    &\Psh^\lrcorner(\mathsf{Cell}~\C\text{-}\mathsf{alg}) \ar[r, "\widehat{\mathrm{B}}_{\iota,!}"]
    & \FMP^c_{\Omega \C} 
    \\
    \mathsf{dg}~\C^*\text{-}\mathsf{alg}~[\mathrm{Q.iso}^{-1}]
    \ar[u, "\mathrm{Y}"] 
    &\mathsf{dg}~\mathcal{C}\text{-}\mathsf{alg}^{\mathsf{qp}\text{-}\mathsf{comp}}~[\mathrm{W}^{-1}] \ar[u, "\mathrm{Y}"] \ar[l, "\Res",swap]
    & \catdgcog{\Omega\C}~\invqis~.
 	 \ar[u, "\mathrm{Y}"] \ar[l, "\widehat{\Omega}_{\iota}",swap]
\end{tikzcd}
\]
\end{proof}

\begin{definition}[Homotopy complete algebra]
Let $\mathfrak{g}$ be a dg $\C^*$-algebra. It is \textit{homotopy complete} if the derived unit of adjunction
\[
\mathbb{L} \eta_{\mathfrak{g}}: \mathfrak{g} \qi \mathrm{Res}~\mathbb{L}\mathrm{Ab}~\mathfrak{g}
\]

is a quasi-isomorphism.
\end{definition}

\begin{remark}
The definition of \textit{homotopy completeness} is analogous to the one of Harper and Hess in \cite{HarperHess}. Over a characteristic zero field, homotopy complete algebras are also the convergence radius of the Goodwillie tower of the identity functor in the category of dg $\C^*$-algebras. One can then reinterpret the condition $(2)$ in Theorem \ref{Thm: the canonical adjunction} as: cellular $\C^*$-algebras are homotopy complete. 

\medskip

The notion of homotopy completeness is deeply linked to \textit{nilpotency}. For example, nilpotent dg Lie algebras in non-negative degrees are homotopy complete as it was shown in \cite[Section 5]{campos2020lie}. The other way around, if a quasi-free dg Lie algebra generated in non-negative degrees is homotopy complete, its homology is a pro-nilpotent graded Lie algebra, see \cite{nilpotentLiestuff}.  
\end{remark}

\begin{corollary}
Let $\mathfrak{g}$ be a homotopy complete dg $\C^*$-algebra. The formal moduli problem $\Psi(\mathfrak{g})$ associated to $\mathfrak{g}$ is given by the functor

\[
\begin{tikzcd}[column sep=1.5pc,row sep=0pc]
\Psi(\mathfrak{g}): \mathsf{coArt}~\Omega \C\text{-}\mathsf{coalg}  \arrow[r]
&\mathcal{S} \\
C \arrow[r,mapsto]
&\mathrm{Map}_{\mathsf{dg}~\C\text{-}\mathsf{alg}~[\mathrm{W}^{-1}]}\left(\widehat{\Omega}_\iota C,\mathbb{L}\mathrm{Ab}(\mathfrak{g})\right)~.
\end{tikzcd}
\] 
\end{corollary}

\begin{proof}
We have that $\mathfrak{g} \simeq \mathrm{Res}~\mathbb{L}\mathrm{Ab}(\mathfrak{g})$ since it is homotopy complete. Furthermore, $\mathbb{L}\mathrm{Ab}(\mathfrak{g})$ is always weakly-equivalent to the complete cobar construction of a dg $\Omega \C$-coalgebra since the complete bar-cobar adjunction is a Quillen equivalence.
\end{proof}

The bright point about Theorem \ref{thm: rectification of fmps} is that it is possible to give explicit models for the mapping spaces that describe the formal moduli problem $\Psi(\mathfrak{g})$ when it is corepresentable by a dg $\Omega \C$-coalgebra. The simplicial sets that model these mapping spaces will have as $0$-simplices solutions to a Maurer-Cartan equation, and the higher simplicies will be given by combinatorial formulae. In particular, we obtain that such models exist for any formal moduli problem encoded by a \textit{homotopy complete} dg $\C^*$-algebra. This is related to the various (pro)-nilpotent conditions that are present in integration theory in the literature, see for instance \cite{Getzler09}. 

\medskip

Let us describe how to construct such models. Since $\C$ is quasi-planar, the dg operad $\Omega\C$ admits a canonical $\mathcal{E}$-comodule structure, where $\mathcal{E}$ is the Barratt-Eccles dg operad, see \cite[Section 2.8]{premierpapier}. This implies that given a dg $\Omega\C$-coalgebra $C$ and a dg $\C$-algebra $A$, there exists a \textit{convolution} curved $\mathrm{B}\mathcal{E}$-algebra structure on the graded module of graded maps $\mathrm{hom}(C,A)$, given by explicit formulas related to the map $\Omega \C \longrightarrow \Omega \C \otimes \mathcal{E}$. We refer to \cite{grignou2022mapping} for the general theory of mapping coalgebras and convolution algebras.

\medskip

A curved $\mathrm{B}\mathcal{E}$-algebra can also be called a curved absolute partition $\mathcal{L}_\infty$-algebras. Using the integration functor $\mathcal{R}$ for curved absolute partition $\mathcal{L}_\infty$-algebras constructed in \cite{lietheoryp}, one can get a rectification of the functor $\Psi(\mathfrak{g})$. It is induced by the following functor 

\[
\begin{tikzcd}[column sep=1.5pc,row sep=0.5pc]
\psi(\mathfrak{g}): \mathsf{dg}~\Omega\C\text{-}\mathsf{coalg} \arrow[r]
&\mathsf{sSets} \\
C \arrow[r,mapsto]
&\mathcal{R}(\mathrm{hom}(C ,\mathbb{L}\mathrm{Ab}(\mathfrak{g})))
~, 
\end{tikzcd}
\]

when restricted to models for coArtinian $\Omega\C$-coalgebras. Furthermore, the assignment is functorial in homotopy complete dg $\C^*$-algebras: any quasi-isomorphisms induces a natural weak-equivalence of functors. See also \cite[Section 2.5]{lietheoryp}.

%----------------------------------------------------------

%%%%%%%%%%%%%%%%%%%%%%%%%%%%%%%%%%%%%%%%%%%%%%%%%%%%%%%%%%

\section{Reduced and tempered cooperads}\label{section: tempered cooperads}
For the rest of this section let $\C$ be a $0$-\textit{reduced} quasi-planar conilpotent dg cooperads. By $0$-reduced, we mean that $\C(0) = 0.$ We give a powerful criterion that turns the functor relating the $\infty$-categories of cellular $\C$-algebras and cellular $\C^*$-algebras an equivalence. In its essence, it holds if the homology of $\C$ is concentrated in increasingly positive homological degrees. The reason behind is that cellular $\C$-algebras and cellular $\C^*$-algebras are equivalent if a graded direct sum is equivalent to a graded direct product, which will be true if there are finitely many non-trivial elements in each homological degree.

\subsection{Reduction to the free case for zero reduced cooperads}
Recall that by Theorem \ref{Thm: the canonical adjunction}, in order to have an equivalence of $\infty$-categories between formal moduli problems over $\Omega\C$-algebras and dg $\C^*$-algebras, it is enough to check that the (derived) unit of the adjunction $\mathbb{L}\mathrm{Ab} \dashv \mathrm{Res}$ is a quasi-isomorphism for every strictly cellular dg $\C^*$-algebra. When $\C$ is zero-reduced, it is in fact possible to reduce even further to the case of free dg $\C^*$-algebras.

\begin{proposition}\label{propositionzeroreducedcase}
As the conilpotent dg cooperad $\C$ is quasi-planar and $0$-reduced, the following assertions are equivalent.
    \begin{enumerate}
        \item For every strictly cellular dg $\C^\ast$-algebra the unit of the $\mathrm{Ab} \dashv \mathrm{Res}$ adjunction is a quasi-isomorphism.
        
        \item For every free dg $\C^*$-algebra generated by a finite dimensional dg module in non-positve degrees, the unit of the $\mathrm{Ab} \dashv \mathrm{Res}$ adjunction is a quasi-isomorphism.
     \end{enumerate}
\end{proposition}

\begin{proof}
The first point $(1)$ clearly implies the second point $(2)$. Let $\left(\C^* \circ V, d_{~\C^* \circ V} + d_2 \right)$ be a strictly cellular dg $\C^*$-algebra. We will first reduce to the case where $d_2$ is concentrated in arity one, that is $d_2(V) \subseteq \C^*(1) \otimes V$, before reducing to the case where $d_2$ is trivial, that is, when the stricly cellular algebra is in fact free.

\medskip

We consider the following decreasing exhaustive filtrations
\[
\mathrm{F}_{m}~\C^* \circ V \coloneqq \bigoplus_{n \geq m} \left(\mathcal{C}(n)^* \otimes V^{\otimes n}\right)^{\mathbb{S}_n}~, \quad \quad
\mathrm{F}_{m}~(V)^{\C} \coloneqq \prod_{n \geq m} \mathrm{Hom}\left(\mathcal{C}(n), V^{\otimes n}\right)^{{\mathbb{S}_n}}~,
\]
for all $m \geq 0$. Since $\C$ is $0$-reduced, both filtrations are stable by the differentials. The unit map given by the natural inclusion 
\[
\eta: \bigoplus_{n \geq m} \left(\mathcal{C}(n)^* \otimes V^{\otimes n}\right)^{\mathbb{S}_n} \rightarrowtail \prod_{n \geq m} \mathrm{Hom}\left(\mathcal{C}(n), V^{\otimes n}\right)^{{\mathbb{S}_n}}
\]
preserves these filtrations. Notice that $\eta$ induces an isomorphism at the first page of the induced spectral sequences. Indeed, since $V$ is a finite dimensional dg module, there is an isomorphism 
\[
\left(\mathcal{C}(n)^* \otimes V^{\otimes n}\right)^{\mathbb{S}_n} \cong \mathrm{Hom}\left(\mathcal{C}(n), V^{\otimes n}\right)^{{\mathbb{S}_n}}~.
\] 
So in order to prove that $\eta$ is a quasi-isomorphism, it is enough to show that both spectral sequences strongly converge. The filtration on the infinite product is complete and Hausdorff, therefore it strongly converges.

\medskip

Let us show that the spectral sequence induced by the filtration $\mathrm{F}_{m}~\C^* \circ V$ is bounded at page two. We suppose that the unit of adjunction is a quasi-isomorphism for $(\C^* \circ V, d_{~\C^* \circ V} + d_2^{\leq 1})$, where $d_2^{\leq 1}$ is the differential $d_2$ truncated at arity one. This unit is given in homology by
\[
\bigoplus_{n \geq 0} \mathrm{H}_*\left(\left(\mathcal{C}(n)^* \otimes V^{\otimes n}\right)^{\mathbb{S}_n},d_{~\C^* \circ V} + d_2^{\leq 1}\right) \rightarrowtail \prod_{n \geq 0} \mathrm{H}_*\left(\left(\mathcal{C}(n)^* \otimes V^{\otimes n}\right)^{\mathbb{S}_n},d_{~\C^* \circ V} + d_2^{\leq 1}\right)~.
\]
It is an isomorphism if and only if for every $k$ in $\mathbb{Z}$, there exists a $m_k$ such that 
\[
\mathrm{H}_k\left(\left(\mathcal{C}(n)^* \otimes V^{\otimes n}\right)^{\mathbb{S}_n},d_{~\C^* \circ V} + d_2^{\leq 1}\right) = 0~,
\]
for all $n \geq m_k$. This implies that the spectral sequence induced by the filtration $\mathrm{F}_{m}~\C^* \circ V$ is bounded at page two.

\medskip

We are left to show that if the unit map is a quasi-isomorphism for $(\C^* \circ V, d_{~\C^* \circ V})$, then it implies that is a quasi-isomorphism for $(\C^* \circ V, d_{~\C^* \circ V} + d_2^{\leq 1})$. It suffices to show that if 
\[
\mathrm{H}_k\left(\left(\mathcal{C}(n)^* \otimes V^{\otimes n}\right)^{\mathbb{S}_n},d_{~\C^* \circ V}\right) = 0 \quad \text{then} \quad \mathrm{H}_k\left(\left(\mathcal{C}(n)^* \otimes V^{\otimes n}\right)^{\mathbb{S}_n},d_{~\C^* \circ V} + d_2^{\leq 1}\right) = 0~. 
\]
Recall that $V$ is endowed with an exhaustive \textit{finite} filtration $V(i)$ since $\C^* \circ V$ is strictly cellular. We consider the following increasing filtration 
\[
\mathrm{F}_i\left(\mathcal{C}(n)^* \otimes V^{\otimes n}\right)^{\mathbb{S}_n} \coloneqq \left(\mathcal{C}(n)^* \otimes \left(\sum_{i_1 + \cdots + i_l = i} V(i_1) \otimes \cdots \otimes V(i_l) \right)\right)^{\mathbb{S}_n}~.
\]
The filtration is finite, therefore the spectral sequence associated it is strongly convergent. Using this spectral sequence, it is straightforward to obtain the desired vanishing result for $\mathrm{H}_k$. 
\end{proof}

\subsection{Tempered cooperads}

\begin{definition}[Tempered cooperad]
Let $\C$ be a quasi-planar $0$-reduced conilpotent dg cooperad. It is \textit{tempered} if for all $\kappa \geq 0$, there exist an $n_\kappa$ such that, for all $n \geq n_\kappa$ and $k \leq \kappa$, $\mathrm{H}_k(\C(n)) = 0$. 
\end{definition}

\begin{remark}
Equivalently, for all $\kappa \geq 0$, there exist an $n_\kappa$ such that, for all $n \geq n_\kappa$ and $k \geq \kappa$, $\mathrm{H}_k(\C^\ast(n)) = 0$. Essentially, it means that the homology of $\C$ is concentrated in increasingly bigger degrees as the arity goes to infinity.
\end{remark}

\begin{remark}[Comparison with splendid operads]\label{Rmk: comparison tempered and splendid}
Let $\kk$ be a field of characteristic zero and let $\mathcal{P}$ be a connective \textit{splendid} dg operad in the sense of \cite{CCN19}. One can always transfer an $\infty$-operad structure onto $\mathrm{H}\mathcal{P}$, and there is a weak-equivalence of conilpotent dg cooperads $\mathrm{B}\mathcal{P} \qi \mathrm{B}\mathrm{H}\mathcal{P}$. Here $\mathrm{B}\mathrm{H}\mathcal{P}$ is a tempered conilpotent dg cooperad concentrated in degrees $\geq 1$. This gives a tempered model for any splendid dg operad.

\medskip

On the other hand, if a conilpotent dg cooperad $\C$ is \textit{tempered}, then one can check that the dg operad $\Omega\C$ is splendid. See \cite[Variant 5.18]{CCN19} for more details. Notice that \textit{op.cit.} uses a \textit{cohomological} convention.
\end{remark}

\begin{lemma}\label{lemma: key lemma}
Let $\C$ be a tempered $0$-reduced quasi-planar conilpotent dg cooperad and let $n \geq 0$. Suppose that $\mathrm{H}_*(\C(n))$ is concentrated in degrees $> \kappa$. Let $V$ be a finite dimensional dg module concentrated in degrees $\leq 0$. The homology 
\[
\mathrm{H}_*\left(\mathrm{Hom}\left(\mathcal{C}(n) , V^{\otimes n} \right)^{\mathbb{S}_n}\right)
\]
is concentrated in degrees $< -\kappa$.
\end{lemma}

\begin{proof}
Let us denote $\tau_{\kappa+1}(\mathcal{C}(n))$ the "intelligent" truncation of $\C(n)$ at degree $\kappa+1$ given by
\[
\tau_{\kappa+1}(\mathcal{C}(n))_k = 
\begin{cases}
    \mathcal{C}(n)_k \text{ if } k \geq \kappa+1~,
    \\
    0 \text{ if } k \leq \kappa~,
    \\
    \mathrm{Ker}(d : \mathcal{C}(n)_{\kappa+1} \to \mathcal{C}(n)_{\kappa+1}) \text{ if } k = \kappa+1~.
\end{cases}
\]
There is a quasi-isomorphism
\[
\tau_{\kappa+1}(\mathcal{C}(n)) \qi  \C(n)~.
\]
The $\mathbb S_n$-module $\C(n)$ is $\mathbb{S}_n$-projective since $\C$ is quasi-planar. In general $\tau_{\kappa+1} \C(n)$ is not a $\mathbb{S}_n$-projective any more. Nevertheless, we can take a cofibrant resolution $P$ of it in the model category of dg $\mathbb S_n$-module in degrees strictly above $\kappa$ endowed with the projective model structure. It is also cofibrant when seen as an unbounded dg $\mathbb S_n$-module for the projective model strucutre. Thus we get a composite quasi-isomorphism of $\mathbb{S}_n$-projective dg $\mathbb S_n$-modules $P \qi \C(n)$.

\medskip

The model category of dg $\mathbb S_n$-module with the projective model structure is homotopically enriched-tensored-cotensored over the dg $\mathbb{S}_n$-modules with the injective model structure. By applying the functor $\mathrm{Hom}\left(- , V^{\otimes n} \right)$ we get a quasi-isomorphism
\[
\mathrm{Hom}\left(\mathcal{C}(n) , V^{\otimes n} \right) \qi \mathrm{Hom}\left(P , V^{\otimes n} \right)
\]

of $\mathbb{S}_n$-injective dg $\mathbb{S}_n$-modules. This give a quasi-isomorphism between their respective $\mathbb{S}_n$-invariants since these also compute the homotopy invariants. But in degree $-\kappa$ and above, the dg module $\mathrm{Hom}_{\mathbb{S}_n}\left(P , V^{\otimes n} \right)$ is zero, therefore so is its homology.
\end{proof}

\begin{theorem}\label{thm: main theorem}
Let $\C$ be a tempered $0$-reduced quasi-planar conilpotent dg cooperad. The canonical adjunction of $\infty$-categories

\[
\begin{tikzcd}[column sep=5pc,row sep=5pc]
\mathsf{FMP}_{\Omega\C} \arrow[r,"\Phi"{name=LDC},shift left=1.1ex ] 
&\mathsf{dg}~\C^*\text{-}\mathsf{alg}~[\mathrm{Q.iso}^{-1}]~,\arrow[l,"\Psi"{name=TD},shift left=1.1ex] \arrow[phantom, from=TD, to=LDC, , "\dashv" rotate=90] 
\end{tikzcd}
\] 
\vspace{0.1pc}
between the $\infty$-categories of formal moduli problems over $\Omega \C$-algebras and the $\infty$-category of dg $\C^*$-algebras is an equivalence.
\end{theorem}

\begin{proof}
As a consequence of Lemma \ref{lemma: key lemma}, the (derived) unit map is a quasi-isomorphism for every dg $\C^\ast$-algebra freely generated by a finite dimensionsal dg module in non-positive degrees. We conclude by Proposition \ref{propositionzeroreducedcase}.
\end{proof}

%%%%%%%%%%%%%%%%%%%%%%%%%%%%%%%%%%%%%%%%%%%%%%%%%%%%%%%%%%%%%%%%%%%%%%%%

\section{Examples}\label{Section: Examples}
\subsection{The characteristic zero case}
If $\kk$ is a field of characteristic zero, the constructions performed so far simplify: we do not need to assume that $\C$ is quasi-planar any more, since every cooperad $\C$ is equivalent to a quasi-planar one. In this context, the methods developed so far provide a new proof of the main results of D. Calaque, R. Campos and J. Nuiten in \cite{CCN19}. See Remark \ref{Rmk: comparison tempered and splendid} for a comparison between tempered cooperads and splendid operads. 

\begin{theorem}[{\cite[Theorem 1.3]{CCN19}}]\label{thm: CCN19}
Let $\mathcal{P}$ be a connective splendid dg operad. There is a canonical equivalence
\[
\mathsf{FMP}_{\mathcal{P}} \simeq \mathsf{dg}~(\mathrm{B}\mathcal{P})^*\text{-}\mathsf{alg}~[\mathrm{Q.iso}^{-1}]~.
\]
\vspace{0.1pc}

between the $\infty$-categories of formal moduli problems over $\mathcal{P}$-algebras and of dg $(\mathrm{B}\mathcal{P})^*$-algebras when localized at quasi-isomorphisms. 
\end{theorem}

While their results provided inspiration for the ones obtained in this paper, it should be noted that their arguments are very similar to the original arguments of J. Lurie \cite{Lurie11}. In particular, the condition that $\mathcal{P}$ is splendid (which is equivalent to $\mathrm{B}\mathcal{P}$ being tempered) only appears as a technical side-condition in order to prove the prerequisites that make the $\infty$-categorical machinery constructed in \cite{Lurie11} run. On the other hand, from the perspective of the new proof given in this paper, this condition is conceptually explained by the comparison between cellular $\mathrm{B}\mathcal{P}$-algebras and cellular $(\mathrm{B}\mathcal{P})^*$-algebras. 

\medskip

\textbf{Lurie--Pridham Theorem.} Let us consider $\mathcal{P}$ to be the dg operad $\mathcal{C}\mathrm{om}$, which encodes non-unital commutative algebras. In this case the conilpotent dg cooperad obtained is $\mathrm{B} \mathcal{C}\mathrm{om}$, and its linear dual $\Omega \mathcal{C}\mathrm{om}^*$ encodes shifted $\mathcal{L}_\infty$-algebras. 

\begin{lemma}\label{prop: sLie tempered}
The dg cooperad $\mathrm{B}\mathcal{C}\mathrm{om}$ is tempered. 
\end{lemma}

\begin{proof}
It is quasi-isomorphic to the shifted Lie cooperad $s\mathcal{L}\mathrm{ie}^*$, which is tempered by direct inspection.
\end{proof}

\begin{theorem}[Lurie-Pridham]
There is a canonical equivalence 

\[
\mathsf{FMP}_{\mathcal{C}\mathrm{om}} \simeq \mathsf{dg}~\mathsf{Lie}\text{-}\mathsf{alg}~[\mathrm{Q.iso}^{-1}]~.
\]
\vspace{0.1pc}

between the $\infty$-categories of formal moduli problems and of dg Lie algebras. 
\end{theorem}

\begin{proof}
Follows from Theorem \ref{thm: main theorem} and Lemma \ref{prop: sLie tempered}.
\end{proof}

The heart of this new approach to the celebrated Lurie--Pridham theorem can be condensed in the fact that the canonical inclusion of dg Lie algebras
\[
\mathcal{L}\mathrm{ie}(V) \to \widehat{\mathcal{L}\mathrm{ie}}(V)
\]
between the free Lie algebra $\mathcal{L}\mathrm{ie}(V)$ and the completed free Lie algebra $\widehat{\mathcal{L}\mathrm{ie}}(V)$ on a finite dimensional dg module concentrated in degree $\leq 1$ is an isomorphism. 

\begin{remark}
The results of \cite{CCN19} have an extra level of generality with respect to our current work. They are valid over coherent homologically bounded dg unital commutative $\kk$-algebra $A$ (this generalization was already achieved by B. Hennion in \cite{Hennion} in the context of formal moduli problems of commutative algebras). Moreover, $A$ can also be a dg category satisfying analogue conditions, which heuristically corresponds to "adding colours" to a dg operad in dg $A$-modules. We nevertheless expect that the methods introduced in this paper can be generalized to a similar setting, as the only obtruction is to generalize the operadic methods to this setting, which was already partially done in \cite{CCN19}. 
\end{remark}

\subsection{Partition Lie algebras}
Let $\kk$ be field of any characteristic. We consider formal moduli problems over an $\mathbb{E}^{\mathrm{nu}}_\infty$ operad, that is, an operad which is a $\mathbb{S}$-projective resolution of the operad $\mathcal{C}\mathrm{om}$. This corresponds to spectral formal moduli problems over $\mathrm{H}\kk$-modules; we emphasize that do not consider the \textit{derived} version here.

\medskip

It was shown by L. Brantner and A. Mathew in \cite{brantnermathew} that the $\infty$-category of these formal moduli problems is equivalent to the $\infty$-category of \textit{(spectral) partition Lie algebras}. The $\infty$-category of partition Lie algebras was first defined to be the $\infty$-category of algebras over a monad. The question of whether this $\infty$-category could admit a simple presentation in terms of an algebraic $1$-category localized at some class of equivalences was settled by L. Brantner, R. Campos and J. Nuiten in \cite{pdalgebras}.

\begin{theorem}[{\cite[Proposition 4.43]{pdalgebras}}]
Let $\mathbb{E}$ be a connected $\mathbb{E}^{\mathrm{nu}}_\infty$-operad. There is a canonical equivalence 
\[
\mathsf{dg}~\Omega(\mathbb{E}^*)\text{-}\mathsf{alg}~[\mathrm{Q.iso}^{-1}] \simeq \mathcal{L}ie^\pi\text{-}\mathsf{alg}
\]
\vspace{0.1pc}

between the $\infty$-categories of dg $\Omega(\mathbb{E}^*)$-algebras and the $\infty$-category of partition Lie algebras. 
\end{theorem}

Therefore, the semi-model category of dg $\Omega(\mathbb{E}^*)$-algebras is a model for partition Lie algebras, and \textit{a fortiori}, their localized $\infty$-category is equivalent to the $\infty$-category of formal moduli problems of $\mathbb{E}^{\mathrm{nu}}_\infty$-algebras over a field $\kk$ of any characteristic. 

\medskip

On the other hand, the methods introduced in this paper give a direct proof that there is an equivalence of $\infty$-categories between the $\infty$-category of formal moduli problems of $\mathbb{E}^{\mathrm{nu}}_\infty$-algebras and the $\infty$-category of dg $\Omega(\mathbb{E}^*)$-algebras. For that, we start by choosing a specific model, the non-unital Barratt-Eccles dg operad $\mathcal{E}^{\mathrm{nu}}$. 

\begin{lemma}\label{prop: Bar de Barratt-Eccles tempered}
The conilpotent dg cooperad $\mathrm{B}(\mathcal{E}^{\mathrm{nu}})$ is quasi-planar and tempered. 
\end{lemma}

\begin{proof}
For the fact that $\mathrm{B}(\mathcal{E}^{\mathrm{nu}})$ is quasi-planar, see \cite[Section 2]{premierpapier}. It follows from \cite[Theorem 6.8]{FressePartitionPoset} that the homology of $\mathrm{B}(\mathcal{E}^{\mathrm{nu}})$ is the underlying $\mathbb{S}$-module of the shifted Lie cooperad $s\mathcal{L}\mathrm{ie}^*$.
\end{proof}

\begin{theorem}\label{thm: Barratt-Eccles encodes}
Let $\mathcal{E}^{\mathrm{nu}}$ be the non-unital Barratt-Eccles dg operad. There is a canonical equivalence 
\[
\mathsf{FMP}_{\mathcal{E}^{\mathrm{nu}}} \simeq \mathsf{dg}~\Omega((\mathcal{E}^{\mathrm{nu}})^*)\text{-}\mathsf{alg}~[\mathrm{Q.iso}^{-1}]~.
\]
\vspace{0.1pc}

between the $\infty$-categories of dg $\Omega((\mathcal{E}^{\mathrm{nu}})^*)$-algebras and the $\infty$-category of formal moduli problems of $\mathcal{E}^{\mathrm{nu}}$-algebras. 
\end{theorem}

\begin{proof}
Follows from Theorem \ref{thm: main theorem} and Lemma \ref{prop: Bar de Barratt-Eccles tempered}.
\end{proof}

This proof has the advantage of directly providing point-set models.

\begin{remark}
The result of Theorem \ref{thm: Barratt-Eccles encodes} can be subsequently extended to other connected models for of $\mathbb{E}^{\mathrm{nu}}_\infty$-operads by the homotopy-invariance of the $\infty$-categories considered. 
\end{remark}

\begin{remark}
The results of \cite{brantnermathew} (resp. of \cite{pdalgebras}) have an extra level of generality with respect to our current work since their results are also valid when $\kk$ is a complete local ring (resp. a coherent $\mathbb{E}_\infty$-ring). We nevertheless expect that the methods introduced in this paper can be generalized provided that the operadic methods of \cite{premierpapier} are extending to this setting.
\end{remark}

\subsection{$\mathbb{E}_k$ formal moduli problems}
Over a field of any characteristic $\kk$, one recovers the correspondence between formal moduli problems over $\mathbb{E}_k^{\mathrm{nu}}$-algebras and $\mathbb{E}_k^{\mathrm{nu}}$-algebras, where $\mathbb{E}_k^{\mathrm{nu}}$ is a model for the non-unital little $k$-cubes operad. For $k=1$, this gives that formal moduli problems of non-unital associative algebras are encoded by non-unital dg associative algebras. See \cite[Section 3 and 4]{Lurie11}. 

\medskip

Let $\mathcal{E}^{\mathrm{nu}}$ be the non-unital Barratt-Eccles dg operad. It admits a filtration of sub dg operads  

\[
\mathcal{E}^{\mathrm{nu}}_1 \rightarrowtail \cdots \rightarrowtail \mathcal{E}^{\mathrm{nu}}_k \rightarrowtail \cdots \rightarrowtail \mathrm{colim}_{k} ~\mathcal{E}^{\mathrm{nu}}_k \cong \mathcal{E}^{\mathrm{nu}}~,
\]
\vspace{0.1pc}

where each $\mathcal{E}_k^{\mathrm{nu}}$ is a model for $\mathbb{E}_k^{\mathrm{nu}}$. See \cite[Section 1.6]{BergerFresse} for more details on this. 

\begin{lemma}\label{lemma: E_k tempered and quasi-planar}
Let $k \geq 1$. The conilpotent dg cooperad $\mathrm{B}\mathcal{E}_k^{\mathrm{nu}}$ is tempered and quasi-planar. 
\end{lemma}

\begin{proof}
Let us notice that for any $k \geq 1$, every $\mathcal{E}_k^{\mathrm{nu}}(n)$ is a quasi-free dg $\mathbb{S}_n$-module. Therefore the underlying conilpotent graded cooperad $\mathrm{B}\mathcal{E}_k^{\mathrm{nu}}$ is planar. Filtering by the homological degree of $\mathcal{E}^{\mathrm{nu}}_k$ together with the weight of rooted trees, gives the quasi-planar filtration on $\mathrm{B}\mathcal{E}^{\mathrm{nu}}_k$. See \cite[Section 2]{premierpapier} for an analogue filtration.

\medskip

It was shown in \cite{FresseEn} that the homology of $\mathrm{B}\mathcal{E}^{\mathrm{nu}}_k$ is given by the underlying $\mathbb{S}$-module of the $k$-fold suspension of the conilpotent dg cooperad $(k-1)\text{-}\mathcal{G}\mathrm{erst}^*$, which is the linear dual cooperad of the operad encoding non-unital $(k-1)$-Gerstenhaber algebras. It follows directly that $\mathrm{B}\mathcal{E}^{\mathrm{nu}}_k$ is tempered.
\end{proof}

\begin{theorem}\label{thm: E_k encodes}
Let $k \geq 1$ and let $\mathcal{E}^{\mathrm{nu}}_k$ the sub dg operad of the non-unital Barratt-Eccles dg operad encoding $\mathbb{E}^{\mathrm{nu}}_k$-algebras. There is a canonical equivalence 

\[
\mathsf{FMP}_{\mathcal{E}^{\mathrm{nu}}_k} \simeq \mathsf{dg}~\Omega ((\mathcal{E}^{\mathrm{nu}}_k)^*)\text{-}\mathsf{alg}~[\mathrm{Q.iso}^{-1}]~.
\]
\vspace{0.1pc}

between the $\infty$-categories of dg $\Omega ((\mathcal{E}^{\mathrm{nu}}_k)^*)$-algebras and of formal moduli problems over $\mathcal{E}^{\mathrm{nu}}_k$-algebras.
\end{theorem}

\begin{proof}
Follows from Theorem \ref{thm: main theorem} and Lemma \ref{lemma: E_k tempered and quasi-planar}.
\end{proof}

\begin{remark}
The result of Theorem \ref{thm: E_k encodes} can be subsequently extended to other connected models for of $\mathbb{E}_k^{\mathrm{nu}}$-operads by the homotopy-invariance of the $\infty$-categories considered. 
\end{remark}

\subsection{Partition pre-Lie algebras and permutative formal moduli problems}
Over a field of characteristic zero $\kk$, the $\infty$-category of formal moduli problems over permutative algebras is equivalent to the $\infty$-category of pre-Lie algebras. This was shown in \cite{CCN19}, by applying Theorem \ref{thm: CCN19} to the operad $\mathcal{P}\mathrm{erm}$ which encodes permutative algebras. See \cite[Section 13.4]{LodayVallette} for more information on permutative algebras and pre-Lie algebras. The main example of formal moduli problems of permutative algebras is given by operadic deformation theory: given a coaugmented dg cooperad $\C$ and a dg operad $\mathcal{P}$, one can define 
\[
\left\{
\begin{tikzcd}[column sep=4pc,row sep=0.5pc]
\mathrm{Def}_{\Omega\C \rightarrow \mathcal{P}}^0: \mathsf{Art}~\mathcal{P}\mathrm{erm}\text{-}\mathsf{alg} \arrow[r]
&\mathcal{S} \\
A \arrow[r,mapsto] 
& \mathrm{Map}_{\mathsf{Op}}\left(\Omega \C, \mathcal{P} \otimes A \right)
\end{tikzcd}
\right.
\]

where $\mathcal{P} \otimes A$ is again a dg operad, since $A$ is a dg permutative algebra. This formal moduli problem encodes the deformations of the trivial morphism $\Omega \C \longrightarrow \mathcal{P}$. When $\mathcal{P}$ is the endomorphism operad $\mathrm{End}_V$ of some dg module $V$, deformations of this trivial morphism correspond to deformations of the trivial dg $\Omega \C$-algebra structure on $V$ (which fix the underlying chain complex $V$). Deformations of other operad morphisms can be considered by taking the fiber at a given morphism. Thus, more generally, one can consider in this way deformation of non-trivial algebraic structures which fix the underlying complex. These deformations can be related to general deformations of algebraic structures which do not fix the underlying complex, see \cite{ginot2019derived}.

\medskip

The permutative formal moduli problem $\mathrm{Def}_{\Omega\C \rightarrow \mathcal{P}}^0$ is encoded by the convolution dg pre-Lie algebra 
\[
\prod_{n \geq 0} \mathrm{Hom}_{\mathbb{S}_n}(\C(n), \mathcal{P}(n))~,
\]

where the pre-Lie bracket is constructed using the partial decomposition of $\C$ and the partial composition of $\mathcal{P}$. Operad morphisms correspond to Maurer--Cartan elements in this pre-Lie algebra; the twisted algebra by a given Maurer--Cartan element encodes deformations of the given morphism. 

\medskip

Over a field of positive characteristic, the operad $\mathcal{P}\mathrm{erm}$ is not $\mathbb{S}$-projective. Therefore in order to make sense of a permutative formal moduli problem, we can consider $\mathcal{P}\mathrm{erm} \otimes \mathcal{E}^{\mathrm{nu}}$, where $\mathcal{E}^{\mathrm{nu}}$ is the non-unital Barratt-Eccles operad. 

\begin{lemma}\label{lemma: BPerm is tempered}
The conilpotent dg cooperad $\mathrm{B}(\mathcal{P}\mathrm{erm} \otimes \mathcal{E}^{\mathrm{nu}})$ is tempered.
\end{lemma}

\begin{proof}
It follows from \cite[Theorem 1.13]{ChapotonVallette} that the homology of $\mathrm{B}(\mathcal{P}\mathrm{erm} \otimes \mathcal{E}^{\mathrm{nu}})$ (which is quasi-isomorphic to $\mathrm{B}(\mathcal{P}\mathrm{erm})$) is the underlying $\mathbb{S}$-module of the shifted pre-Lie cooperad $s\mathrm{pre}\text{-}\mathcal{L}\mathrm{ie}^*$. Thus it clearly satisfies the tempered condition.
\end{proof}

\begin{theorem}
There is a canonical equivalence 

\[
\mathsf{FMP}_{\mathcal{P}\mathrm{erm} ~\otimes ~\mathcal{E}^{\mathrm{nu}}} \simeq \mathsf{dg}~\Omega(\mathcal{P}\mathrm{erm}^* \otimes (\mathcal{E}^{\mathrm{nu}})^*)\text{-}\mathsf{alg}~[\mathrm{Q.iso}^{-1}]~.
\]
\vspace{0.1pc}

between the $\infty$-categories of dg $\Omega(\mathcal{P}\mathrm{erm}^* \otimes (\mathcal{E}^{\mathrm{nu}})^*)$-algebras and of formal moduli problems over $\mathcal{P}\mathrm{erm} ~\otimes ~\mathcal{E}^{\mathrm{nu}}$-algebras.
\end{theorem}

\begin{proof}
Follows from Theorem \ref{thm: main theorem} and Lemma \ref{lemma: BPerm is tempered}.
\end{proof}

\begin{remark}
Adapting the arguments of \cite[Proposition 4.51]{pdalgebras}, one can show that $\mathcal{S}\mathrm{urj}^* \otimes s\mathrm{pre}\text{-}\mathcal{L}\mathrm{ie}$ is also a model for $\Omega(\mathcal{P}\mathrm{erm}^* \otimes (\mathcal{E}^{\mathrm{nu}})^*)$, where $\mathcal{S}\mathrm{urj}^*$ is the dual of the surjections cooperad constructed in the appendix of \textit{op.cit}.
\end{remark}

In order to define the analogue of the functor $\mathrm{Def}_{\Omega\C \rightarrow \mathcal{P}}^0$ in positive characteristic, one can replace dg operads by some kinds of operads up to homotopy ; for instance dg algebras over the coloured operads $\mathcal{O}\mathrm{p} \otimes \mathcal{E}^{\mathrm{nu}}$ (where $\mathcal{O}\mathrm{p}$ is the coloured operad that encodes non-unital dg operads). These dg operads up to homotopy are tensored over $\mathcal{P}\mathrm{erm} \otimes \mathcal{E}^{\mathrm{nu}}$-algebras. For two operads $\mathcal P, \mathcal Q$, we can then define a formal moduli problem over $\mathcal{P}\mathrm{erm} \otimes \mathcal{E}^{\mathrm{nu}}$-algebras in an analogous way
\[
\begin{tikzcd}[column sep=2.5pc,row sep=0pc]
A \arrow[r,mapsto]
&\Map{\catdgalg{\mathcal{O}\mathrm{p} ~\otimes~ \mathcal{E}^{\mathrm{nu}}}}{\mathcal Q}{\mathcal P \otimes A}~.
\end{tikzcd}
\]

Replacing $\mathcal Q$ by the cobar construction $\Omega \C$ of a quasi-planar conilpotent dg cooperad $\C$, the convolution construction between $\C$ and $\PP$ should encode this formal moduli problem, and therefore should be endowed with a dg partition pre-Lie algebra structure that extends the classical pre-Lie structure.

%----------------------------------------------------------

\bibliographystyle{alpha}
\bibliography{bibax}
\end{document}